\documentclass[11pt, final]{amsart}
\usepackage[english]{babel}
\usepackage{latexsym}
\usepackage{amssymb}
\usepackage{amscd}
\usepackage[cp850]{inputenc}
\usepackage[mathscr]{eucal}

\newcommand{\cl}{\mathcal}

\newcommand{\ra}{\rightarrow}
\newcommand{\lra}{\longrightarrow}

\theoremstyle{plain}

\newtheorem{probl}{Problem}
\newtheorem{defin}{Definition}
\newtheorem{teor}{Theorem}[section]
\newtheorem{prop}[teor]{Proposition}
\newtheorem{cor}[teor]{Corollary}

\newcommand{\adef}{\begin{defin}}
\newcommand{\zdef}{\end{defin}}

\newtheorem{lemma}[teor]{Lemma}

\theoremstyle{definition}
\newcommand\sgn{\mathop{\mathrm{sgn}}\nolimits}

\newcommand\dist{\mathop{\mathrm{dist}}\nolimits}

\newcommand\supp{\mathop{\mathrm{supp}}\nolimits}
\newcommand{\lop}{\ensuremath{\curvearrowright}}

\theoremstyle{remark}

\newcommand{\R}{\mathbb{R}}
\newcommand{\N}{\mathbb{N}}

\def\Ext{\operatorname{Ext}}

\newcommand{\seq}{\ensuremath{0\to Y\to X\to Z\to 0\:}}

\newcommand{\aproof}{\begin{proof}}
\newcommand{\zproof}{\end{proof}}

\newcommand\PB{{\mathrm{PB}}}
\newcommand\PO{{\mathrm{PO}}}

\newcommand{\pp}{{p^{\prime}}}

\author{Jesus M. F. Castillo, Valentin Ferenczi and Manuel Gonz\'alez}

\address{Departamento de Matem\'aticas, Universidad de Extremadura,
Avenida de Elvas s/n, 06011 Badajoz, Spain}
\email{castillo@unex.es}

\address{Departamento de Matem\'atica, Instituto de Matem\'atica e
Estat\'\i stica, Universidade de S\~ao Paulo, rua do Mat\~ao 1010,
05508-090 S\~ao Paulo SP, Brazil  \\ and \newline
Equipe d'Analyse Fonctionnelle \\
Institut de Math\'ematiques de Jussieu \\
Universit\'e Pierre et Marie Curie - Paris 6 \\
Case 247, 4 place Jussieu \\
75252 Paris Cedex 05 \\
France.}
\email{ferenczi@ime.usp.br}

\address{Departamento de Matem\'aticas, Universidad de Cantabria,
Avenida de los Castros s/n, 39071 Santander, Spain}
\email{manuel.gonzalez@unican.es}

\thanks{This research was supported by Project MTM2013-45643, D.G.I. Spain.
The second author was supported by Fapesp project 2013/11390-4, including for
visits of the first and third author to the University of S\~ao Paulo.}

\begin{document}

\title{Singular twisted sums generated by complex interpolation}

\begin{abstract}
We present new methods to obtain singular twisted sums $X\oplus_\Omega X$
(i.e., exact sequences $0\to X\to X\oplus_\Omega X \to X\to 0$ in which
the quotient map is strictly singular), in which  $X$ is the interpolation space arising
from a complex interpolation scheme and $\Omega$ is the induced centralizer.

Although our methods are quite general, in our applications we are mainly concerned
with the choice of $X$ as either a Hilbert space, or Ferenczi's uniformly convex
Hereditarily Indecomposable space.
In the first case, we construct new singular twisted Hilbert spaces, including the only known example so far: the Kalton-Peck space $Z_2$. In the second case we obtain the first example of an H.I. twisted sum of an H.I. space. We then use Rochberg's description of iterated twisted sums to show that there is a
sequence $\mathcal F_n$ of H.I. spaces so that $\mathcal F_{m+n}$ is a singular
twisted sum of  $\mathcal F_m$ and $\mathcal F_n$, while for $l>n$ the direct sum
$\mathcal F_n \oplus \mathcal F_{l+m}$ is a nontrivial twisted sum of $\mathcal F_l$
and $\mathcal F_{m+n}$.

We also introduce and study the notion of disjoint singular twisted sum of K\"othe function spaces and construct several examples involving reflexive $p$-convex K\"othe function spaces, which include the function version of the  Kalton-Peck space $Z_2$. \end{abstract}

\date{}

\maketitle

\thispagestyle{empty}

\section{Introduction}
\label{sect:intro}

For all unexplained notation see Sections \ref{sect:twisted-sums}
(background on exact sequences and quasi-linear maps) and \ref{sect:centralizers} (background and preliminary results on complex interpolation and centralizers).

This paper focuses on the study of the existence and properties of exact sequences
\begin{equation}\label{twix}
\begin{CD}
0@>>> X@>j>> E @>q>> X @>>>0,
\end{CD}
\end{equation}
in which the Banach space $X$ has been obtained by complex interpolation.
The exact sequence is called {\it nontrivial} when $j(X)$ is not complemented in the
middle space $E$, which is then called a (nontrivial) twisted sum of $X$ (or a twisting
of $X$, or even a twisted $X$).
The exact sequence is called {\it singular} (and $E$ is called a \emph{singular twisted sum})
when the operator $q$ is strictly singular.
The key example on which all the theory is modeled is the Kalton-Peck twisted Hilbert
space $Z_2$ obtained in \cite{kaltpeck}, which provides the first and only known singular
sequence
$$
\begin{CD}
0@>>> \ell_2@>j>> Z_2 @>q>> \ell_2 @>>>0.
\end{CD}
$$

Singular sequences correspond to twisted sums which are, in some sense, as far as
possible from being direct sums. For example, in Kalton-Peck example $Z_2$, the natural copy of $\ell_2$ does not even admit
a "relative" summand, i.e. there is no infinite dimensional subspace $Z$ of $Z_2$ forming
a topological direct sum $ j(\ell_2) \oplus Z$ inside $Z_2$.

In \cite{kalt} Kalton showed that exact sequences (\ref{twix}) are in correspondence
with certain non-linear maps $F: X\to  X$, called quasi-linear maps. So,  twisted sum spaces, and in particular exact sequences, can be written in the form
\begin{equation}\label{quasi}\begin{CD}
0@>>> X@>>> X\oplus_F X @>>> X @>>>0.
\end{CD}\end{equation}

Following \cite{ccs,castmorestrict}, we say that a quasi-linear map $F$ is {\it singular}
if the associated exact sequence (\ref{quasi}) is singular.
In \cite{kaltpeck} Kalton and Peck presented a method to show an explicit
construction of quasi-linear maps on Banach spaces with unconditional basis. This method was refined by Kalton \cite{kaltmem} and extended to K\"othe function spaces. The special type of quasi-linear maps obtained by this method were called {\it centralizers}. The main examples are the so called {\em Kalton-Peck maps:}
$$
\mathscr K_\phi(x) = x \phi \left( - \log \frac{|x|}{\|x\|}\right)
$$
where $\phi:\R\to \R$ is a certain Lipschitz map. The choice of the function $\phi_r(t)=t^r$ (when $t\geq 1$), and $\phi_r(t)=t$ (when $0\leq t\leq 1$); with $0<r\leq 1$ will have a especial interest for us.
We simply write $\mathscr K$ for Kalton-Peck space instead of $\mathscr K_{\phi_1}$.
In \cite{kaltpeck} it is shown that $\mathcal K$ is singular on $\ell_p$ spaces for
$1<p<\infty$; in \cite{castmorestrict} for $p=1$; and in \cite{ccs} for the whole range
$0<p<\infty$.
It was soon observed that the Kalton-Peck map $\mathscr K$ on $\ell_2$ could be generated from the interpolation scale of $\ell_p$ spaces. Taking this as starting point, Kalton unfolds in \cite{kaltmem,kaltdiff} the existence of
a correspondence between centralizers defined on K\"othe function spaces and interpolation
scales of K\"othe function spaces. This opens the door to the possibility of obtaining nontrivial quasi-linear maps in Banach
spaces generated by an interpolation scale, even when no unconditional structure is present.\\

Such is the point of view we adopt in this paper to tackle the study of singular centralizers
and singular quasi-linear maps on Banach spaces obtained by complex interpolation.
In the case of centralizers this leads us to obtain new singularity results for Kalton-Peck
sums of sequence spaces as well as of function spaces; and, in particular, new singular twisted
Hilbert spaces. We introduce a new concept of singularity, that we call \emph{disjoint singularity},
which is relevant to the study of interpolation schemes of function spaces.
In the case of general quasi-linear maps, not just centralizers, we ``localize" the techniques developed and apply
them to spaces with monotone basis and obtain the first H.I. twisted sum of an H.I. space.\\

A description of the contents of the paper is in order:
after this introduction and a preliminary Section \ref{sect:twisted-sums} on basic facts
about exact sequences and quasi-linear maps, Section \ref{sect:centralizers} takes root in
Kalton's work and so it contains an analysis of centralizers arising from an interpolation
scheme; the analysis is centered on an interpolation couple $(X_0,X_1)$ and the centralizer
$\Omega_\theta$ obtained at the interpolation space $X_\theta = (X_0,X_1)_\theta$; although
the results  extend (see subsection \ref{interp-family}) to cover the case of a measurable
family of spaces.
We observe, and derive a few consequences from it, the fact that such centralizers admit an
overall form as $\Omega_\theta(x) = x\log \frac{a_0(x)}{a_1(x)}$, where
$a_0(x)^{1-\theta} a_1(x)^\theta$ is a Lozanovskii factorization of $|x|$ with respect to
the couple $(X_0, X_1)$. Section \ref{sect:estimates} contains the two fundamental estimates we use through the
paper: Lemma \ref{trivialestimate} (estimate for non-singular maps) and Lemma \ref{core}
(general estimate for centralizers arising from an interpolation scheme).
Section \ref{sect:singular} contains several criteria for singularity based on the previous
two lemmata.
The first two subsections treat the Banach lattice case:
we recover the singularity of Kalton-Peck maps associated to the interpolation scale of
$\ell_p$ spaces, as a particular case of a general criterion for singularity, and we prove
the disjoint singularity of Kalton-Peck maps associated to the
interpolation scale of $L_p$ spaces, for which it was known that they were not singular.
We also prove the disjoint singularity of Kalton-Peck maps on more general $p$-convex
K\"othe spaces, by means of the interpolation formula $X=(L_\infty,X^{(p)})_{1/p}$.
In the third subsection, we give conditions implying the singularity in the conditional case
(spaces admitting a basis not necessarily unconditional) which will be needed to cover the
case of H.I. spaces.
In Section \ref{sect:twisting-Hilbert} we obtain new singular twisted Hilbert spaces;
we also complete previous results by showing that centralizers $\mathscr K_\phi$ is singular under rather mild conditions on $\phi$, satisfied in particular by the complex versions \cite{kalt-canad} of $\mathscr K$.
In Section \ref{sect:twisting-H.I.} we connect the results about singular sequences with
the twisting of H.I. spaces: a twisted sum of two H.I. spaces is H.I. if and only if it
is singular. One of the difficulties for such construction is, as we show, that a
nontrivial twisted sum of two H.I. spaces can be decomposable; note that it was known \cite[Theorem 1]{gonzherre} that such twisted sums should be at most 2-decomposable. Section \ref{sect:twisting-Ferenczi} applies the previous techniques to the quasi-linear map associated to the construction of Ferenczi's H.I. space $\mathcal F$ \cite{fere} by
complex interpolation of a suitable family of Banach spaces.
In Section \ref{sect:higher-twisting} we complete and improve the results in
Sections \ref{sect:twisting-H.I.} and \ref{sect:twisting-Ferenczi} by showing new natural H.I. and decomposable twistings; precisely, that there is a sequence $(\mathcal F_n)$ of H.I. spaces so that:
\begin{itemize}

\item[(i)] For each $m,n\geq 1$ there is a singular exact sequence
$$\begin{CD}
0@>>> \mathcal F_m @>>> \mathcal F_{m+n} @>>> \mathcal F_n @>>>0.
\end{CD}$$
\item[(ii)] For each $l,m,n\geq 1$ with $l>n$ there is a nontrivial exact sequence
$$\begin{CD}
0@>>> \mathcal F_l @>>> \mathcal F_{n} \oplus \mathcal F_{l+m} @>>> \mathcal F_{m+n} @>>>0.
\end{CD}$$
\end{itemize}

\section{Exact sequences, twisted sums and quasi-linear maps}
\label{sect:twisted-sums}

A {\it twisted sum} of two Banach spaces $Y$ and $Z$ is a space $X$ which has a subspace
$M$ isomorphic to $Y$ with the quotient $X/M$ isomorphic to $Z$.
The space $X$ is a quasi-Banach space in general \cite{kaltpeck}.
Recall that a Banach space is \emph{B-convex} when it does not contain $\ell_1^n$
uniformly.
Theorem 2.6 of \cite{kalt} implies that a twisted sum of two B-convex Banach spaces is
isomorphic to a Banach space.

An exact sequence $0 \to Y \to X \to Z \to 0$, where $Y,Z$ are Banach spaces and the
arrows are (bounded) operators is a diagram in which the kernel of each arrow coincides
with the image of the preceding one.
By the open mapping theorem this means that the middle space $X$ is a twisted sum
of $Y$ and $Z$.

Two exact sequences $0 \to Y \to X_1 \to Z \to 0$ and $0 \to Y \to X_2 \to Z \to 0$
are {\it equivalent} if there exists an operator $T:X_1\to X_2$ such that the following
diagram commutes:
$$
\begin{CD}
0 @>>>Y@>i>>X_1@>q>>Z@>>>0\\
&&@| @VVTV @|\\
0 @>>>Y@>>j>X_2@>>p>Z@>>>0
\end{CD}$$
The classical 3-lemma (see \cite[p. 3]{castgonz}) shows that $T$ must be an isomorphism.
An exact sequence is trivial if and only if it is equivalent to
$0 \to Y \to Y \times Z \to Z \to 0$, where $Y\times Z$ is endowed with the product norm.
In this case we say that the exact sequence \emph{splits.}

A map $F:Z\to Y$ is called \emph{quasi-linear} if it is homogeneous and there is a
constant $M$ such that $\|F(u+v)-F(u)-F(v)\|\leq M\|u+v\|$ for all $u,v\in Z$.
There is a correspondence (see \cite[Theorem 1.5.c, Section 1.6]{castgonz}) between exact
sequences $0 \to Y \to X \to Z \to 0$ of Banach spaces and a special kind of quasi-linear
maps $\omega:Z\to Y$, called $z$-linear maps, which satisfy
$\|\omega(\sum_{i=1}^n u_i) - \sum_{i=1}^n \omega(u_i)\| \leq M \sum_{i=1}^n \|u_i\|$
for all finite sets $u_1,\ldots, u_n \in Z$.
A quasi-linear map $F : Z \to Y$ induces the exact sequence
$0 \to Y \stackrel{j}\to Y \oplus_F Z \stackrel{p}\to Z \to 0$ in which
$Y \oplus_F Z$ denotes the vector space $Y \times Z$ endowed with the quasi-norm
$\|(y,z) \|_F = \|y - F(z)\| + \|z\|$.
The embedding is  $j(y)= (y,0)$ while the quotient map is $p(y,z)=z$.
When $F$ is $z$-linear, this quasi-norm is equivalent to a norm \cite[Chapter 1]{castgonz}.
On the other hand, the process to obtain a $z$-linear map out from an exact sequence
$0\to Y \stackrel{i}\to X \stackrel{q}\to Z \to 0$ of Banach spaces is the following one:
get a homogeneous bounded selection $b: Z \to X$ for the quotient map $q$, and then
a linear $\ell: Z \to X$ selection for the quotient map.
Then $\omega= i^{-1}(b - \ell)$ is a $z$-linear map from $Z$ to $Y$. The commutative diagram
$$
\begin{CD}
0 @>>>Y@>i>>X@>q>>Z@>>>0\\
&&@| @VVTV @|\\
0 @>>>Y@>>j>Y\oplus_\omega Z@>>p>Z@>>>0
\end{CD}$$
obtained by taking as $T:X\to  Y\oplus_\omega Z$ the operator $T(x) = (x - \ell qx, qx)$
shows that the upper and lower exact sequences are equivalent.
Two quasi-linear maps $F, F': Z \to Y$ are said to be equivalent, denoted $F\equiv G$,
if the difference $F-F'$ can be written as $B +L$, where $B: Z \to Y$ is a homogeneous bounded
map (not necessarily linear) and $L: Z \to Y$ is a linear map (not necessarily bounded).
Of course two quasi-linear maps are equivalent if and only if the associated exact sequences
are equivalent.
Thus, two exact sequences
$$\begin{CD}
0@>>> Y @>>> Y\oplus_\Omega Z @>>> Z @>>> 0\\
&&@|&&@|\\
0@>>> Y @>>> Y\oplus_\Psi Z @>>> Z @>>> 0\end{CD}
$$
(or two quasi-linear maps $\Omega, \Psi$) are equivalent (i.e., $\Omega \equiv \Psi$)
if there exists a commutative diagram
$$
\begin{CD}
0 @>>>Y@>>>Y\oplus_\Omega Z @>>>Z@>>>0\\
&&@V{\alpha}VV @V{\beta}VV @VV{\gamma}V\\
0 @>>>Y@>>>Y \oplus_\Psi Z@>>>Z@>>>0
\end{CD}
$$
with $\alpha=id_Y$ and $\gamma=id_Z$.
Imposing other conditions on the maps $\alpha, \beta, \gamma$ yields other notions of
equivalence appearing in the literature. From the most restrictive to the more general they are:

\begin{enumerate}
\item Bounded equivalence \cite{kaltmem,kaltdiff} (see also Section 3 below): asking that
$\Omega -\Psi$ is bounded.
\item Projective equivalence \cite{kaltpeck}: asking $\alpha, \gamma$ to be scalar
multiples of the identity.
Equivalently, $\Omega \equiv \mu\Psi$ for some scalar $\mu$.
\item We will need in this paper ``permutative projective equivalence": when $Y$ and
$Z$ have unconditional bases $(e_n)$, asking  $T_\sigma \Omega \equiv \mu \Psi T_\sigma$
for some scalar $\mu$  and some operator $T_\sigma(\sum_i x_i e_i)=\sum_i x_i e_{\sigma(i)}$
induced by a permutation $\sigma$ of the integers.
When $\mu=1$ we will just say that $\Omega$ and $\Lambda$ are permutatively equivalent.
\item Isomorphic equivalence \cite{cabecastdu,castmoreisr}: asking $\alpha, \beta, \gamma$
to be isomorphisms.
In quasi-linear terms, this means that $\alpha\Omega \equiv \Psi\gamma$.
\end{enumerate}

Obviously, equivalence takes place between (1) and (2).
In conclusion, each of (1), (2), (3), (4) yields a "natural" isomorphism $\beta$ between
$Y\oplus_\Omega Z$ and  $Y \oplus_\Psi Z$ of a specific form prescribed by the forms of
the maps $\alpha$ and $\gamma$.
\medskip

A few facts about the connections between quasi-linear maps and the associated exact
sequences will be needed in this paper, and can be explicitly found
in \cite[Section 1]{castmorebap}.
Given an exact sequence \seq with associated quasi-linear map $F$ and an operator
$\alpha: Y\to Y'$, there is a commutative diagram
\begin{equation}\label{po}
\begin{CD}
0 @>>>Y@>i>>X@>q>>Z@>>>0 \\
&&@V{\alpha}VV @VTVV @|\\
0 @>>>Y'@>i'>> \PO @>q'>>Z@>>>0.\end{CD}
\end{equation}

The lower sequence is called the \emph{push-out sequence,} its associated quasi-linear
map is equivalent to $\alpha\circ F$, and the space $\PO$ is called the
\emph{push-out space.}
When $F$ is $z$-linear, so is $\alpha\circ F$.
Given a commutative diagram like (\ref{po}) the \emph{diagonal push-out sequence}
is the exact sequence generated by the quasi-linear map $F\circ q'$, and is equivalent to
the exact sequence
$$
\begin{CD}
0@>>> Y @>d>> Y'\oplus X @>m>> X' @>>> 0
\end{CD}
$$
\smallskip

\noindent
where $d(y)= (-\alpha y, i y)$ and $m(y',x) = i' y' + Tx$.

\section{Complex interpolation and centralizers} \label{sect:centralizers}

Here we explain the connections between complex interpolation, twisted sums and
quasi-linear maps that we use throughout the paper.

\subsection{Complex interpolation and twisted sums}

We describe the complex interpolation method for a pair of spaces following \cite{Bergh-Lofstrom}.
Other general references are  \cite{CCRSW2,kaltdiff,kaltonmontgomery,rochberg-weiss}.

Let $\mathbb S$ denote the closed strip $\{z\in \mathbb C: 0\leq {\rm Re}(z)\leq 1\}$ in the complex
plane, and let $\mathbb S^\circ$ be its interior and $\partial \mathbb S$ be its boundary.
Given an admissible pair $(X_0,X_1)$ of complex Banach spaces, we denote by $\cl{H}=\cl H(X_0,X_1)$
the space of functions $g:\mathbb S\to \Sigma:=X_0+X_1$ satisfying the following conditions:
\begin{enumerate}
\item $g$ is $\|\cdot\|_\Sigma$-bounded and $\|\cdot\|_\Sigma$-continuous on $\mathbb S$,
and $\|\cdot\|_\Sigma$-analytic on $\mathbb S^\circ$;

\item $g(it)\in X_0$ for each $t\in\R$, and the map $t\in\R\mapsto g(it)\in X_0$ is bounded
and continuous;

\item $g(it+1)\in X_1$ for each $t\in\R$, and the map $t\in\R\mapsto g(it+1)\in X_1$ is bounded
and continuous;
\end{enumerate}

The space $\cl{H}$ is a Banach space under the norm
$\|g\|_\cl{H} = \sup\{\|g(j+it)\|_j: j=0,1; t\in\R \}$.
For $\theta\in[0,1]$, define the interpolation space
$$
X_\theta=(X_0,X_1)_\theta=\{x\in\Sigma: x=g(\theta) \text{ for some } g\in\cl H\}
$$
with the norm $\|x\|_\theta=\inf\{\|g\|_H: x=g(\theta)\}$.
So $(X_0,X_1)_\theta$ is the quotient of $\cl{H}$ by $\ker\delta_\theta$,
and thus it is a Banach space.

For  $0<\theta<1$, we will consider the maps $\delta_\theta^n: \mathcal H\to \Sigma$
--evaluation of the $n^{th}$-derivative at $\theta$-- that appear in Schechter's version
of the complex method of interpolation \cite{schechter}.
Note that $\delta_\theta\equiv \delta_\theta^0$ is bounded by the definition of
$\mathcal{H}$, and this fact and the Cauchy integral formula imply the boundedness
of $\delta_\theta^{n}$ for $n\geq 1$ (see also \cite{cabecastroch}).
We will also need the following result (see \cite[Theorem 4.1]{carro}):

\begin{lemma}\label{mechanism}
$\delta'_\theta: \ker \delta_\theta\to X_\theta$ is bounded and onto for $0<\theta<1$.
\end{lemma}

For future use, note that given $G \in {\rm ker}\ \delta_\theta$, the function $H$ defined
by $H(z)=G(z)/(z-\theta)$ belongs to $\mathcal H$ and satisfies
$\delta^{\prime}_\theta(G)=H(\theta)$,  which implies the estimate
$$
\|{\delta^{\prime}_\theta}_{|{\rm ker}\ \delta_\theta}\|\leq \dist(\theta,\partial\mathbb S)^{-1}.
$$

Lemma \ref{mechanism} provides the connection with exact sequences and twisted sums
through the following push-out diagram:

\begin{equation}\label{basic}\begin{CD}
0@>>> \ker\delta_\theta @>i_\theta>> \mathcal H @>\delta_\theta >> X_\theta @>>>0\\
&&@V\delta_\theta' VV @VVV @| \\
0@>>> X_\theta @>>> \PO @>>> X_\theta@>>> 0
\end{CD}
\end{equation}
whose lower row is obviously a twisted sum of $X_\theta$.

Apart from the obvious description as a push-out space, $\PO$ can be represented as:
\medskip


(1) \emph{A twisted sum space.}
Let $B_\theta: X_\theta \to \mathcal H$ be a bounded homogeneous selection for $\delta_\theta$,
and let $L_\theta: X_\theta \to \mathcal H$ be a linear selection.
The map $\omega_\theta= B_\theta -L_\theta: X_\theta \to \ker\delta_\theta$ is an associated
quasi-linear map for the upper sequence in diagram (\ref{basic}).
The lower push-out sequence then comes defined by the quasi-linear map
$\delta_\theta' \omega_\theta.$
Hence, $\PO \simeq X_\theta \oplus_{\delta_\theta' \omega_\theta} X_\theta.$
\medskip

(2) \emph{A derived space.} With the same notation as above, set
$$
d_{\delta_\theta' B_\theta}(X_\theta) = \{(y,z)\in \Sigma  \times X_\theta\, :\,
 y - \delta_\theta' B_\theta z\in X_\theta\}
$$
endowed with the quasi-norm
$\|(y,z)\|_d = \|y -\delta_\theta' B_\theta z\|_{X_\theta} +\|z\|_{X_\theta}$.
This is a twisted sum of $X_\theta$ since the embedding $y\to (y,0)$ and
quotient map $(y,z) \to z$ yield an exact sequence
$$
\begin{CD}
0@>>> X_\theta @>>> d_{\delta_\theta' B_\theta}(X_\theta) @>>> X_\theta@>>>0.
\end{CD}
$$
Moreover, the two exact sequences
$$
\begin{CD}
0@>>> X_\theta@>>> X_\theta \oplus_{\delta_\theta' \omega_\theta} X_\theta@>>> X_\theta @>>>0\\
&& \parallel  && @VVTV \Vert\\ 0@>>> X_\theta @>>> d_{\delta_\theta' B_\theta}(X_\theta)
@>>> X_\theta@>>>0.
\end{CD}
$$
are isometrically equivalent via the isometry $T(y,z) =(y +\delta_\theta' L_\theta z, z)$.
\medskip

Thus, we can pretend that the quasi-linear map associated to the push-out sequence
is $\delta_\theta' B_\theta$, usually more intuitive than the true quasi-linear map
$\delta_\theta'(B_\theta- L_\theta)$.
Such map has been sometimes called ``the $\Omega$-operator".
Needless to say, the $\Omega$-operator depends on the choice of $B_\theta$.
However the difference between two associated $\Omega$-operators must be bounded:
$$
\|\delta'_\theta(\tilde B_\theta - B_\theta)x\|_{X_\theta}\leq
\|{\delta_\theta'}_{|\ker \delta_\theta}\| (\|\tilde B_\theta\| +\|B_\theta\|)\|x\|_{X_\theta}.
$$

The derived space admits the following useful representation, see \cite[p.323]{rochberg-weiss}
for an embryonic finite-dimensional version; also quoted in \cite[p.218]{CJMR};
see \cite[Prop.7.1]{carro} for a general version involving two compatible interpolators,
and \cite{cabecastroch} for a rather complete exposition, variations and applications
of that representation.

\begin{prop}\label{rep-derived-space}
\quad $d_{\delta_\theta' B_\theta}(X_\theta) = \{\big(f'(\theta), f(\theta)\big)\,
:\, f\in  \mathcal{H} \}$,\quad
and the quotient norm of $\mathcal H/(\ker \delta_\theta \cap \ker \delta'_\theta)$ is
equivalent to the quasi-norm $\|(\cdot,\cdot)\|_d$.
\end{prop}
\begin{proof}
Given $f\in \mathcal{H}$, since $f-B_\theta( f(\theta)) \in \ker \delta_\theta$, by
Lemma \ref{mechanism} one has
$$
f'(\theta)- \delta'_\theta B_\theta(f(\theta)) = \delta'_\theta(f-B_\theta (f(\theta)))
\in X_\theta,
$$
hence $\big(f'(\theta), f(\theta)\big)\in d_{\delta_\theta' B_\theta}(X_\theta)$.

Conversely, let $(y,z)\in d_{\delta_\theta' B_\theta}(X)$.
We have $z\in X_\theta$, so $B_\theta z \in\mathcal{H}$.
Since $y - \delta_\theta' B_\theta z\in X_\theta$, there exists $g\in \ker \delta_\theta$
such that $y - \delta_\theta' B_\theta z = g'(\theta)$.
Thus taking $f= B_\theta z+g$ we have $f(\theta)=z$ and $f'(\theta)=y$,
and the equality is proved.

For the equivalence, given $(y,z)\in d_{\delta_\theta' B_\theta}(X)$, take
$f\in \mathcal{H}$ with $\|f\|\leq 2\dist(f,\ker \delta_\theta\cap\ker\delta'_\theta)$
such that $y= f'(\theta)$ and $z=f(\theta)$.
Then $\|z\|_{X_\theta}=\dist(f,\ker \delta_\theta)$ and
$$
\|y -\delta_\theta' B_\theta z\|_{X_\theta}=\|\delta_\theta' (f-B_\theta z)\|_{X_\theta}.
$$
Since $f-B_\theta z \in\ker \delta_\theta$, we get
$$
\|(y,z)\|_d \leq \|{\delta_\theta'}_{|\ker \delta_\theta}\| (1+\|B_\theta\|) \|f\| + \|f\|
\leq 2(\|{\delta_\theta'}_{|\ker \delta_\theta}\|(1 + \|B_\theta\|)+1)
\dist(f,\ker \delta_\theta\cap\delta'_\theta),
$$
and there exists a constant $C$ so that
$\dist(f,\ker \delta_\theta\cap\delta'_\theta)\leq C \|(y,z)\|_d$
by the open-mapping theorem.
\end{proof}

The results mentioned so far remain valid in the wider context of
the general method of interpolation considered in \cite[Section 10]{kaltonmontgomery}.
In Section \ref{sect:twisting-Ferenczi} we will need to work with the complex interpolation
method associated to a family $(X_{(0,t)},X_{(1,t)})_{t\in \R}$ of complex Banach spaces
as described in  \cite{CCRSW2}; which is a special case of the general method mentioned above.


\subsection{Centralizers.}

Here we consider K\"othe function spaces $X$ over a measure space $(\Sigma,\mu)$ with their
$L_\infty$-module structure.
As a particular case, we have Banach spaces with a $1$-unconditional basis with their
associated $\ell_\infty$-structure. We denote by $L_0$ the space of all $\mu$-measurable
functions, and given $g\in L_0$, we understand that $\|g\|_X<\infty$ implies $g\in X$.

\begin{defin}
A centralizer on a K\"othe function space $X$ is a homogeneous map $\Omega: X \lop L_0$
such that $\|\Omega(ax) -a\Omega(x)\|_X\leq C\|x\|_X \|a\|_\infty$ for all $a\in L_\infty$
and $x\in X$.
\end{defin}

A centralizer on $X$ will be denoted by $\Omega: X\lop X$. We use this notation to stress the fact that a centralizer on $X$ is not a map $X\to X$, but only a map $X\to L_0$ so that the differences $\Omega(ax) -a\Omega(x)$ belong to $X$. This notion coincides with that of Kalton's ``strong centralizer" introduced in \cite{kaltmem}. Centralizers arise naturally in a complex interpolation scheme in which the interpolation scale
of spaces share a common $L_\infty$-module structure: in such case, the space $\mathcal H$
also enjoys the same $L_\infty$-module structure in the form $(u\cdot f)(z)=u\cdot f(z)$.
In this way, the fundamental sequence of the interpolation scheme
$0\to \ker \delta_\theta \to \mathcal H\to X_\theta \to 0$
is an exact sequence in the category of $L_\infty$-modules.
In an interpolation scheme starting with a couple $(X_0,X_1)$ of K\"othe function spaces,
the map $\Omega_\theta = {\delta_\theta}' B_\theta$ is a centralizer on $X_\theta$.\\

For a centralizer $\Omega: X\lop X$ on a K\"othe function space $X$, it was proved
in \cite[Lemma 4.2]{kaltmem} that there exists $M>0$ such that
$\|\Omega(u+v) -\Omega(u)-\Omega(v)\|_X\leq M(\|u\|_X + \|v\|_X)$.
So we can assume that $\Omega$ is a quasi-linear map.
The smallest of the constants $M$ satisfying the above inequality is denoted $\rho(\Omega)$.
Note that $\Omega: X\lop X$ induces an exact sequence in the category of
(quasi-)Banach $L_\infty$-modules $0\to X\to d_\Omega(X)\to X\to 0$, where
$$
d_\Omega(X)=\{(w,z)\in L_0 \times X :\,  w-\Omega z\in X\}
$$
endowed with the quasi-norm $\|(w,z)\|_\Omega = \|w- \Omega z\|_X +\|z\|_X$;
with embedding $y\to (y,0)$ and quotient map $(w,z)\to z$.
The derived space $d_\Omega(X)$ admits a $L_\infty$-module structure defined by
$a(w,z) = (aw, az)$.
Kalton proved in \cite[Section 4]{kaltmem} that every self-extension of a K\"othe function
space $X$ is (equivalent to) the extension induced by a centralizer on $X$.
Sometimes we will take the restriction of $\Omega$ to a closed subspace $Y$ of $X$,
and consider $d_\Omega(X,Y)$ defined in the same way as a subspace of $L_0 \times Y$.\\

A centralizer $\Omega: X\lop X$ is said to be \emph{bounded} when there exists a constant $C>0$
so that $\|\Omega(x)\|_X\leq C\|x\|_X$ for all $x\in X$; which in particular means that $\Omega(x)\in X$ for all $x\in X$.
Two centralizers $\Omega_1: X\lop X$ and $\Omega_2: X\lop X$ are \emph{equivalent} if and only if
the induced exact sequences are equivalent, which happens if and only if there exists
a linear map $L:X\to L_0$ so that $\Omega_1 -\Omega_2 - L$ is bounded.
Two centralizers $\Omega_1: X\lop X$ and $\Omega_2: X\lop X$ are said to
be \emph{boundedly equivalent} when $\Omega_1-\Omega_2$ is bounded.
The interest in this notion (which, to some extent, plays the role of triviality for quasi-linear
maps) stems from the following outstanding result of Kalton \cite[Theorem 7.6]{kaltdiff}:

\begin{teor}\label{kalton}
Let $X$ be a separable superreflexive K\"othe function space.
Then there exists a constant $c$ (depending on the concavity of a $q$-concave renorming of $X$)
such that if $\Omega: X\lop X$ is a real centralizer on $X$ with $\rho(\Omega)\leq c$, then
\begin{enumerate}
\item There is a pair of K\"othe function spaces $X_0, X_1$ such that $X=(X_0,X_1)_{1/2}$ and
 $\Omega - \Omega_{1/2}$ is bounded.
\item The spaces $X_0, X_1$ are uniquely determined up to equivalent renorming.
\end{enumerate}
\end{teor}

An example is in order: taking the couple $(\ell_1, \ell_\infty)$, the map $B(x) = x^{2(1-z)}$
is a homogeneous bounded selection for the evaluation map $\delta_{1/2}: \mathcal H\to \ell_2$;
hence the interpolation procedure yields the centralizer $-2\mathscr K$;
while the couple $(\ell_p, \ell_{p^*})$ yields $-2(\frac{1}{p} -\frac{1}{p^*})\mathscr K$.
As we see the two centralizers are the same up to the scalar factor.
Theorem \ref{kalton} shows however that the scalar factor cannot be overlooked since it actually
determines the end points $X_0, X_1$ in the interpolation scale. See the general situation in Proposition \ref{reitera}.\\

We note for future use that the condition on $\rho(\Omega)$, which is necessary for existence,
is not necessary for uniqueness; thus, uniqueness may be stated as follows:

\begin{prop}\label{kaltonunique}
Let $X$ be a separable superreflexive K\"othe function space.
Assume that $X=(X_0,X_1)_\theta=(Y_0,Y_1)_\theta$, where $0<\theta<1$ and $X_i,Y_i$ are K\"othe
function spaces, and that the associated maps $\Omega_X$ and $\Omega_Y$ are boundedly equivalent.
Then $X_0=Y_0$ and $X_1=Y_1$.
\end{prop}

\begin{proof}
We follow Kalton's notation and the first steps of the proof of uniqueness in Kalton's
theorem \cite[Theorem 7.6]{kaltdiff}, which is written in the case $\theta=1/2$.
Since $\Omega_X$ and $\Omega_Y$ are boundedly equivalent, $\Omega_X^{[1]}$ and $\Omega_Y^{[1]}$
are boundedly equivalent.
Hence on a suitable strict semi-ideal, $\Phi^{\Omega_X}$ is equivalent to $\Phi_{Y_1}-\Phi_{Y_0}$,
while $(1-\theta) \Phi_{Y_0}+\theta \Phi_{Y_1}$ is equivalent to $\Phi_X$.
Thus, up to equivalence $\Phi_{Y_0}$ and $\Phi_{Y_1}$ are uniquely determined.
\cite[Proposition 4.5]{kaltdiff} shows then that the spaces $Y_0$ and $Y_1$ are unique up to
equivalence of norm.
\end{proof}




\subsection{Centralizers and Lozanovskii's decomposition}
\label{subsect:Lozano}\

Here we follow Kalton (see \cite[formula (3.2)]{kaltdiff}) to obtain a formula for the centralizer $\Omega_\theta$ corresponding to the interpolation of a couple of K\"othe function spaces $(X_0,X_1)$. Let $0<\theta<1$, and suppose that one of the spaces $X_0$, $X_1$ has the Radon-Nikodym property.
The Lozanovskii decomposition formula allows us to show (see \cite[Theorem 4.6]{kaltonmontgomery})
that the complex interpolation space $X_\theta$ is isometric to the space $X_0^{1-\theta} X_1^\theta$,
with
$$
\|x\|_\theta=\inf \{\|y\|_0^{1-\theta}\|z\|_1^{\theta}: y\in X_0, z\in X_1,
|x|=|y|^{1-\theta}|z|^{\theta}\}.
$$
By homogeneity we may always assume that $\|y\|_0=\|z\|_1$ for $y,z$ in this infimum.
When $\|y\|_0, \|z\|_1 \leq K\|x\|_\theta$ we shall say that
$|x|=|y|^{1-\theta}|z|^{\theta}$ is a {\em $K$-optimal decomposition} for $x$.
When $x$ is finitely supported or $X$ is uniformly convex a $1$-optimal (or simply, optimal)
decomposition may be achieved. A simple choice of $B_\theta(x)$ can be made for positive $x$
as follows:
Let $a_0(x),a_1(x)$ be a $(1+\epsilon)$-optimal (or optimal when possible) Lozanovskii
decomposition for $x$.
Since $\|x\|_\theta = \|a_0(x)\|_0=\|a_1(x)\|_1$, set $B_\theta(x)\in \mathcal H$ given
by $B_\theta(x)(z) = |a_0(x)|^{1-z} |a_1(x)|^z$.
One thus gets for positive $x$ the formula:
$$
\Omega_\theta(x) = \delta_\theta'B_\theta(x) = |a_0(x)|^{1-\theta} |a_1(x)|^{\theta}
\log \frac{|a_1(x)|}{|a_0(x)|}x =x\, \log \frac{|a_1(x)|}{|a_0(x)|}.
$$

Using $B_\theta(x)=({\rm sgn}\ x) B_\theta(|x|)$ for general $x$ one still gets
\begin{equation}\label{kpgeneralized}
\Omega_\theta(x)  = x\, \log \frac{|a_1(x)|}{|a_0(x)|}.
\end{equation}

Recall that a \emph{unit} $u$ in $L_\infty$ is an element which only takes the values $\pm 1$.
Thus one has:

\begin{lemma}\label{addit-prop}
The centralizer $\Omega_\theta=\delta_\theta'B_\theta$ on $X_\theta=(X_0,X_1)_\theta$
verifies:
 \begin{enumerate}
\item $\Omega_\theta(ux) = u\Omega_\theta(x)$ for every unit $u$ and $x \in X_\theta$;
\item $\supp \Omega_\theta(x)\subset \supp x$ for every $x \in X_\theta$;
\item when $X_0$ and $X_1$ are spaces with a normalized $1$-unconditional basis $(e_n)$,
$\Omega_\theta(e_n)=0$ for all $n$.
\end{enumerate}
\end{lemma}


The Lozanovskii approach can be used to make explicit the Kalton correspondence between
centralizers and interpolation scales in some special cases.
Recall that the $p$-convexification of a K\"othe function space $X$ is defined by the norm
$\||x\||=\| |x|^p\|^{1/p}$.
Conversely, when $X$ is $p$-convex, the $p$-concavification of $X$ is given by
$\||x\||=\| |x|^{1/p}\|^p$.
Modulo the fact that every uniformly convex space may be renormed to be $p$-convex for some
$p>1$, the following proposition interprets Kalton-Peck maps defined on uniformly convex
spaces as induced by specific interpolation schemes.

\begin{prop}\label{spirit}
Let $0<\theta<1<p<\infty$, and let $X$ be a Banach space with $1$-unconditional basis
(respectively a K\"othe function space).
Then
$X_\theta = (\ell_\infty, X)_\theta$ (respectively $(L_\infty(\mu),X)_\theta$) is the
$\theta^{-1}$-convexification of $X$, and the induced centralizer on $X_\theta$ is
$$
\Omega(x)= \theta^{-1} \, x \log(|x|/||x||_\theta).
$$

Conversely if $X$ is $p$-convex and $X^{p}$ is the $p$-concavification of $X$ then
$X=(\ell_\infty,X^{p})_{1/p}$ (respectively $X=(L_\infty(\mu),X^p)_{1/p}$), and the
induced centralizer is defined on $X$ by
$$
\Omega(x)=p\, x \log(|x|/||x||).
$$
\end{prop}

\begin{proof}
We write down the proof for unconditional basis, the other being analogous.
For normalized positive $x$ in $X_\theta$, write $x=a_0(x)^{1-\theta}a_1(x)^{\theta}$
and look for such a (normalized) decomposition which is optimal.
Since $a_0(x) \in \ell_\infty$,  we may assume that $a_0(x)$ has constant coefficients equal
to $1$  on the support of $x$: otherwise, we may increase the non $1$ coordinates
of $a_0(x)$ to $1$, therefore diminishing the corresponding coordinates of $a_1(x)$
and non-increasing the norm of $a_1(x)$ by  $1$-unconditionality, and still get
something optimal.
So $a_0(x)=1_{\rm supp}(x)$ and  $x=a_1(x)^\theta$.
Therefore
$\|x\|_\theta=\|a_1(x)\|^{\theta}=\|x^{1/\theta}\|^\theta$.
So $X_\theta$ is the $\theta^{-1}$-convexification of $X$ and
$$
\Omega_\theta(x)=x \log( a_1(x)/a_0(x) )= \frac{1}{\theta} \, x \log (x).
$$

As for the converse, note that when we interpolate $\ell_\infty$ and some $Y$ we have
$|a_1(x)|=|x|^p$ for $x$ normalized in $Y_\theta$, so if we interpolate $\ell_\infty$
and $Y=X^{(p)}$ we obtain for such $x$
$$
\|x\|_{Y_\theta}=1=\|a_1(x)\|_Y=\||x|^p\|_Y=\| (|x|^p)^\theta \|_X^p =\|x\|_X^p,
$$
therefore $X=Y_\theta=(\ell_\infty,X^{(p)})_\theta$.

The remaining part of the converse is an immediate consequence of the first part of
the proposition.
\end{proof}

As we announced before Theorem \ref{kalton}, we show now the dependence of the scalar factor with respect to different choices of endpoints in a given interpolation scale:

\begin{prop}\label{reitera}
Let $(X_0,X_1)$ be an admissible pair of K\"othe function spaces and for some
$0<\alpha<\beta<1$, consider also the admissible pair $(X_\alpha,X_\beta)$.
Let $\alpha<\theta<\beta$ so that one has $(X_0,X_1)_\theta = (X_\alpha,X_\beta)_\rho$
for some $0<\rho<1$.
Let $\Omega$ (resp. $\Omega'$) denote the centralizers generated by the couple $(X_0,X_1)$
(resp. $(X_\alpha,X_\beta)$).
Then $\Omega'_\rho =(\beta-\alpha)\Omega_\theta$.
\end{prop}
\begin{proof}
It is easy to check (see \cite[Theorem 4.5]{kaltonmontgomery}) that $\rho$ is given by
$\alpha(1-\rho)+\beta \rho=\theta$.
Let us consider the centralizers
$$
\Omega_\theta(x)  = x\, \log \frac{|a_1(x)|}{|a_0(x)|}\quad \textrm{ and }\quad
\Omega'_\rho(x)  = x\, \log \frac{|a_\beta(x)|}{|a_\alpha(x)|}.
$$
Since $x=a_0(x)^{1-\theta}a_1(x)^\theta$, $1-\theta=(1-\alpha)(1-\rho)+(1-\beta)\rho$ and
$\theta=\alpha(1-\rho)+\beta\rho$ we get
$$
x=\big(a_0(x)^{1-\alpha}a_1(x)^\alpha\big)^{1-\rho}\big(a_0(x)^{1-\beta}a_1(x)^{\beta}\big)^\rho.
$$
Thus taking $a_\alpha(x)=a_0(x)^{1-\alpha}a_1(x)^\alpha$ and $a_\beta(x)=a_0(x)^{1-\beta}a_1(x)^\beta$
it is not difficult to check that the minimality of $x=a_0(x)^{1-\theta}a_1(x)^\theta$ implies the
minimality of $x=a_\alpha(x)^{1-\rho}a_\beta(x)^\rho$, and the equality
$\Omega'_\rho(x) =(\beta-\alpha)\Omega_\theta(x)$
follows from the properties of the logarithm function.
\end{proof}

\subsection{The case of Orlicz function spaces} We now describe the centralizers associated to Orlicz function spaces over a measure space $(\Sigma,\mu)$.
Recall that an \emph{$N$-function} is a map $\varphi:[0,\infty)\ra[0,\infty)$ which is strictly
increasing, continuous, $\varphi(0)=0$, $\varphi(t)/t\ra 0$ as $t\ra 0$, and $\varphi(t)/t\ra \infty$
as $t\ra \infty$.
An $N$-function $\varphi$ satisfies the \emph{$\Delta_2$-property} if there exists a number $C>0$
such that $\varphi(2t)\leq C\varphi(t)$ for all $t\geq 0$.
For $1<p<\infty$, $\varphi(t)=t^p$ is $N$-function satisfying the $\Delta_2$-property.

When an $N$-function $\varphi$ satisfies the $\Delta_2$-property, the \emph{Orlicz space}
$L_\varphi(\mu)$ is given by
$$
L_\varphi(\mu) = \{f\in L_0(\mu) : \varphi(|f|)\in L_1(\mu)\}.
$$
with the norm
$$
\|f\|=\inf \{r>0 : \int \varphi(|f|/r) d\mu \leq 1\}
$$
The following result was proved in \cite{Gustavsson}, and a clear exposition can be found
in \cite{CFMMN}.

\begin{prop}\label{interp-Orlicz}
Let $\varphi_0$ and $\varphi_1$ be two $N$-functions satisfying the $\Delta_2$-property, and
let $0<\theta<1$.
Then the formula
$\varphi^{-1} = \big(\varphi_0^{-1}\big)^{1-\theta} \big(\varphi_1^{-1}\big)^\theta$
defines an $N$-function $\varphi$ satisfying the $\Delta_2$-property, and
$\big(L_{\varphi_0}(\mu), L_{\varphi_1}(\mu)\big)_\theta = L_\varphi(\mu)$.
\end{prop}

Next we give an expression for the centralizer associated to a Hilbert space obtained by
complex interpolation of Orlicz spaces.
Note that once we have defined a centralizer $\Omega$ for normalized $0\leq f\in X$,
we can define $\Omega(0)=0$ and $\Omega(g)= g\cdot\Omega(|g|/\|g\|)$ for $0\neq g\in X$.

\begin{prop}\label{interp-Orlicz-central}
Let $\varphi_0$ and $\varphi_1$ be two $N$-functions satisfying the $\Delta_2$-property and
such that $t =\varphi_0^{-1}(t) \cdot \varphi_1^{-1}(t)$.
Then $\big(L_{\varphi_0}(\mu), L_{\varphi_1}(\mu)\big)_{1/2} = L_2(\mu)$ and the induced
centralizer is
$$
\Omega_{1/2}(f)= f\, \log \frac{\varphi_1^{-1}(f^2)}{\varphi_0^{-1}(f^2)} =
2 f\, \log \frac{\varphi_1^{-1}(f^2)}{f}\quad (0\leq f\in L_2(\mu), \|f\|_2=1).
$$
\end{prop}
\begin{proof}
First we consider the general case
$\varphi^{-1}:= \big(\varphi_0^{-1}\big)^{1-\theta} \big(\varphi_1^{-1}\big)^\theta$,
as in Proposition \ref{interp-Orlicz}.
For $0\leq f$ normalized in $L_\varphi(\mu)$  we can write
$f = \big(\varphi_0^{-1}\varphi(f)\big)^{1-\theta} \big(\varphi_1^{-1}\varphi(f)\big)^\theta$.
Thus a selection of the quotient map $\mathcal{H}\ra L_\varphi(\mu)$ is given by
$B_\theta(f)(z) = \big(\varphi_0^{-1}\varphi(f)\big)^{1-z} \big(\varphi_1^{-1}\varphi(f)\big)^z$.
Differentiating
$B_\theta(f)'(z) =B_\theta(f)(z)\, \log\frac{|\varphi_1^{-1}(\varphi(f))|}{|\varphi_0^{-1}\varphi(f)|}$,
hence
$\Omega_{1/2}(f)=B_{1/2}(f)'(1/2)=f\, \log\frac{|\varphi_1^{-1}(\varphi(f))|}{|\varphi_0^{-1}\varphi(f)|}$,
which gives the desired result when $\varphi(t)=t^2$.
\end{proof}

\subsection{Additional properties.} The properties of $\Omega_\theta$ obtained in Lemma \ref{addit-prop} will turn out essential
for our estimates, so they deserve a definition.

\adef
Let $X$ be a K\"othe function space.
A centralizer $\Omega: X\lop X$ is called \emph{exact} if for each $x\in X$ and every unit $u$
one has $\Omega(ux)=u\Omega x$.
It is called \emph{contractive} if $\supp \Omega(x)\subset \supp x$ for every $x\in X$.
\zdef

One has:

\begin{lemma}\label{propers}
Let $X$ be a K\"othe function space.
\begin{enumerate}
\item Every exact quasi-linear map on $X$ is contractive.
\item If $X$ is reflexive, then every exact trivial centralizer $\Omega$ on $X$
admits an exact linear map $\Lambda$ such that $\Omega-\Lambda$ is bounded.
\item If $X$ has unconditional basis $(e_n)$ and is reflexive, and if $\Omega$ is exact
and trivial on $X$, and satisfies $\Omega(e_n)=0$ for all $n$, then $\Omega$ is bounded.
\end{enumerate}
\end{lemma}
\begin{proof}
(1) Let $u \in L_\infty$ be the function with value $1$ on the support of $x$
and $-1$ elsewhere, then $ux=x$, therefore $u\Omega(x)=\Omega(ux)=\Omega(x)$ which means
that $\supp\Omega(x)$ is included in the support of $x$.

(2) Let $\Omega$ be a centralizer with constant $C$ and assume that it is trivial.
So some linear map $\ell: X \to L_0$ exists such that $B:=\Omega - \ell$ is bounded.
Let $U$ denote the abelian group of units in in $L_\infty$.
Then $U$ is amenable, so there exists a left invariant finitely additive mean $m$ on $U$
allowing to define for any bounded $f: U \to \R$ an integral $\int_U f(u) dm$.
Since $X$ is reflexive we may then define for
any bounded  $f: U \to X$ an element  $x=\int_U f(u) dm \in X$ in the natural way, i.e.
$$\phi(x)=\int_U \phi(f(u)) dm$$ for every $\phi \in X^*$.
One can therefore define a map $\Lambda: X\to  L_0$ as follows:
$$
\Lambda(x) =\Omega(x)-  \int_{U} u B(u x) dm.
$$
Then the homogeneous map $x \mapsto \int_{U} u B(u x) dm$ is bounded, and by exactness
of $\Omega$ and invariance of $m$, we have that $\Lambda$ is exact.
It is also easy to check that $\Lambda$ is linear.
Indeed, denoting by  $\Delta(x,y)$ the element $\Omega(x+y)-\Omega x -\Omega y
=B(x+y)-Bx-By \in X$, and observing that $\Delta(ux,uy)=u\Delta(x,y)$, we obtain
\begin{displaymath}
\begin{array}{ll}
\Omega(x+y)-\Lambda(x+y)&=
\int_{U} u B(ux+uy) dm  \\
&= \int_U u\Delta(ux,uy) dm
+\int_U uBux dm + \int_U uBuy dm \\
&=  \Delta(x,y)
+\Omega(x)-\Lambda(x)+\Omega(y)-\Lambda(y) \\
&=\Omega(x+y)-\Lambda(x)-\Lambda(y).
\end{array}
\end{displaymath}

(3) We claim that $\Lambda(x)=ax$ for all $x \in X$, where $\Lambda(e_n)=a_n e_n$.
Indeed
$$\Lambda(x)=\Lambda(x-x_n e_n)+ \Lambda(x_n e_n)=\Lambda(x-x_n e_n)+a_n x_n e_n$$
which, since $\Lambda(x-x_n e_n)$ has support disjoint from $n$, implies that
the $n$-th entry of $\Lambda(x)$ is $a_n x_n$.
Since $\Omega(e_n)=0$, $a_n e_n=-B(e_n)$, and therefore $(a_n)_n$
is a bounded sequence.
So unconditionality applies to make $\Lambda$  bounded.
Since $\Omega -\Lambda$ is also bounded, $\Omega$ is bounded.
\end{proof}

A reformulation of (3) will provide us in due time with a criterion to distinguish between
permutatively projectively equivalent centralizers:

\begin{cor}\label{prepara}
Let $\Omega$ and $\Psi$ be exact centralizers on a reflexive space $X$ with
$1$-unconditional basis $(e_n)$, and such that $\Omega(e_n)=\Psi(e_n)=0$ for all $n\in\N$.
If $\Omega$ and $\Lambda$ are equivalent then they are boundedly equivalent.
\end{cor}
\begin{proof} $\Omega - \Lambda$ is still an exact centralizer vanishing on the $e_n$.
Thus, if it is trivial then it is bounded.
\end{proof}

Lemma \ref{propers} can be generalized for maps between two different modules.
We are interested in the particular  case in which one has to combine two related actions:
let $X$ be an $L_\infty$-Banach module and let $W\subset X$ be a subspace generated by
disjointly supported elements $W=[u_n]$.
Consider in this case the subspace $L_\infty^W\subset L_\infty$ formed by the elements
which are constant on the supports of all the $u_n$. Let $U_W$ be its group of units.
We say that a map $\Omega: W\to X$ is \emph{relatively exact} if $\Omega(ux)=u\Omega(x)$
for all $u\in U_W$ and $x \in W$, and we say that $\Omega$ is \emph{relatively contractive}
if ${\rm supp}_X \Omega(x) \subset {\rm supp}_X x$, for all  $x \in W$.
One has:

\begin{lemma}\label{propersing}
Let $X$ be a K\"othe function space, and let $W$ be a subspace of $X$ generated by
disjointly supported elements.
Then:
\begin{enumerate}
\item  If $\Omega: X\lop X$ is a exact centralizer then the restriction $\Omega_{|W}$
is relatively exact.
\item Every relatively exact map $W \lop X$ is relatively contractive.
\item Assume $X$ is reflexive. If some relatively exact $\Omega: W \lop X$  is trivial
then there exists a relatively exact linear map $\Lambda: W \rightarrow X$ such that
$\Omega-\Lambda$ is bounded.
\end{enumerate}
\end{lemma}

\begin{proof}
Assertion (1) is obvious, (2) has the same proof as before.
For (3), assuming $\Omega=B+\ell$, where $B: W \rightarrow X$ is bounded and
$\ell: W \rightarrow L_0$ is linear, define for $x \in W$,
$$
\Lambda(x) =  \Omega(x)-\int_{U_W} u B(u x) dm,
$$
where $m$ is a left invariant finitely additive mean on $U_W$.
\end{proof}

\begin{lemma}$\;$\begin{enumerate}
\item Every centralizer $\Omega$ on a K\"othe function space admits a exact
centralizer $\omega$ such that $\Omega -\omega$ is bounded.

\item Every exact centralizer (resp. quasi-linear map) $\Omega$ between Banach spaces
with unconditional basis admits a exact centralizer (resp. quasi-linear map) $\omega$
such that $\omega(e_n)=0$ and $\Omega - \omega$ is linear and exact.

\item Every contractive centralizer (resp. quasi-linear map) $\Omega$ between K\"othe
function spaces admits, for every sequence $(f_n)$ of disjointly supported vectors,
a contractive centralizer (resp. quasi-linear map) $\omega$ such that $\omega(f_n)=0$
and $\Omega - \omega$ is linear and contractive.
\end{enumerate}
\end{lemma}
\begin{proof}
Assertion (1) is in \cite[Prop. 4.1]{kaltmem}.
In fact, $\omega(x) = \|x\| \sgn(x) \Omega(|x|/\|x\|)$ for $x\neq 0$.
To prove (2), note that since $\Omega$ is contractive, $\Omega(e_n)=\mu_ne_n$, and we
may define the multiplication linear map $\ell(x)=\mu x$, where $\mu=(\mu_n)_n$.
Thus $\omega=\Omega -\ell$ is the desired map.
To prove (3), define as above a linear map by $\ell(f_n)=\Omega(f_n)$.
If $\Omega$ is contractive, so is $\ell$ and thus $\omega=\Omega -\ell$ is the desired map.
\end{proof}

\section{Singularity and estimates for exact centralizers}
\label{sect:estimates}

Recall that an operator between Banach spaces is said to be \emph{strictly singular}
if no restriction to an infinite dimensional closed subspace is an isomorphism.

\adef
A quasi-linear map (in particular, a centralizer) is said to be \emph{singular} if its
restriction to every infinite dimensional closed subspace is never trivial.
An exact sequence induced by a singular quasi-linear map is called a \emph{singular sequence.}\zdef

It is well known \cite{castmorestrict} that a  quasi-linear map is singular if and only
if the associated exact sequence has strictly singular quotient map. So singular quasi-linear maps induce twisted sums which are, in some sense, as nontrivial as is possible. The following notion is perhaps more suitable to work with K\"othe function spaces.

\adef A quasi-linear map on a K\"othe function space is called \emph{disjointly singular} if
its restriction to every subspace generated by a disjoint sequence is never trivial.
\zdef

One can show that a quasi-linear map  $F: Z \to Y$ is disjointly singular if and only if
the quotient map of the induced exact sequence $0\to Y\to Y\oplus_F Z\stackrel{q}\to  Z\to 0$ is never an isomorphism on a subspace $X$ of  $Y\oplus_F Z$ such that $q(X)$ is generated by disjoint vectors in $Z$. Observe that when $Z$ has an unconditional basis and the lattice structure one considers is the one induced by the basis then the two notions coincide since saying that $q$ is an isomorphism on some subspace is the same that saying that it is an isomorphism on some subspace whose image is generated by blocks of the basis. Thus:

\begin{lemma} A quasi-linear map $F:Z\to Y$ on a Banach space $Z$ with unconditional basis is singular if and only if it is disjointly singular with respect to the induced latttice structure.
\end{lemma}

Singularity implies disjoint singularity and, as we shall see, the reverse implication does not hold in general. Of course, a disjointly singular quasi-linear map is nontrivial. The following ``transfer principle" (\cite{castmorestrict}, \cite{ccs}) will be essential for us.

\begin{lemma}\label{transfer}
If a quasi-linear map defined on a Banach space $X$ with basis is trivial on some infinite
dimensional subspace of $X$ then it is also trivial on some subspace $W=[w_n]$ of $X$ spanned
by normalized blocks of the basis.
\end{lemma}

Observe that if $F$ is a quasi-linear map on a K\"othe space $X$, and if for some
sequence $(u_n)$ of disjointly supported vectors and some constant $K$ one has
$$
\left\| F(\sum \lambda_j u_j) - \sum \lambda_j F(u_j)\right\| \leq
K \|\sum \lambda_j u_j\|
$$
for all choices of scalars $(\lambda_j)$ then $F$ is not singular: indeed,
the estimate above means that the linear map  $[u_j]\to X\oplus_F [u_j]$ given by
$u_j \to (0, u_j)$ is continuous.
Under exactness conditions we can get a partial converse.

\begin{lemma}\label{trivialestimate}
Let $\Omega: X\lop X$ be an exact centralizer on a reflexive K\"othe function space.
If $\Omega$ is not disjointly singular, then there exists a subspace $W$ of $X$ generated by a
disjoint sequence and a constant $K$ such that given vectors $u_1,\ldots, u_n$ in $W$ there
are vectors $z_1, ..., z_n$ in $X$ with $\mathrm{supp} z_i \subset \mathrm{supp} u_i$ and
$\|z_i\|\leq K\|u_i\|$ such that for all scalars $\lambda_1, \dots, \lambda_n$ one has
\begin{equation}\label{esteem}
\|\Omega(\sum_{i=1}^n \lambda_i u_i)-\sum_{i=1}^n \lambda_i \Omega(u_i)\|
\leq K \left( \|\sum_{i=1}^n \lambda_i u_i\| +\|\sum_{i=1}^n \lambda_i z_i\|\right).
\end{equation}
\end{lemma}

\begin{proof}
Since $\Omega$ is not disjointly singular, it is trivial on some subspace $W=[u_n]$
spanned by disjointly supported vectors.
Then by Lemma \ref{propersing} there exists a linear relatively exact map
$\Lambda: W \rightarrow X$ so that $\Omega_{|W} -\Lambda$ is bounded.
Since both $\Omega$ and $\Lambda$ (by Lemma \ref{propersing} (2)) are relatively
contractive, so is $\Omega -\Lambda$.
Set $z_i= (\Omega - \Lambda)(u_i)$ and $K=\|\Omega_{|W}- \Lambda\|$.
\end{proof}

The preceding estimate can be considered as a subtler version of the ``upper
$p$-estimates" argument for non-splitting, which can be quickly described as:
if the space $X$ verifies some type of upper $p$-estimate and the twisted sum
$X\oplus_\Omega X$ splits then the space $X\oplus_\Omega X$ must also verify the upper
$p$-estimate (the key here is the $p$ since, in general,  if $X$ has type $p$ then
$X\oplus_\Omega X$ only needs to have type $p+\varepsilon$ for every $\varepsilon$
(see  \cite{kaltconvex}).
Therefore, given suitable vectors $(u_n)$ in $X$ the elements $(0, u_n)$ in
$X\oplus_\Omega X$ should verify an upper $p$-estimate; which amounts to
$$
\|\Omega(\sum_{i=1}^n  u_i)-\sum_{i=1}^n \Omega(u_i)\| \leq C\sqrt[n]{p}.
$$

We now introduce the notion of standard class of finite families of elements
of K\"othe spaces to simplify the exposition.

\adef
A \emph{standard class} $\mathcal{S}$ is a class of finite families ($n$-tuples) of elements
of K\"othe function spaces (respect. spaces with $1$-unconditional bases) $X$ satisfying
\begin{itemize}
\item[(i)] whenever $(x_i)\in \mathcal{S}$ and $\mathrm{supp} z_i \subset \mathrm{supp} x_i$
for all $i$ then $(z_i)\in \mathcal{S}$;
\item[(ii)] assume that
$W$ is a subspace generated by disjoint vectors (resp. generated by successive vectors) of $X$,
and $(x_i)$ is $n$-tuple of elements of $W$; if $(x_i)$ belongs to $\mathcal{S}$ as a family
in $W$, then it also belongs to $\mathcal{S}$ as a family in $X$.
\end{itemize}
\zdef

The two standard classes we shall use in this paper are disjointly supported vectors in
K\"othe spaces and "Schreier" successive vectors on $1$-unconditional bases, i.e. families
$(x_1,\ldots,x_n)$ such that $n < {\rm supp}\ x_1 < \cdots <{\rm supp}\ x_n$, but some other
examples like successive vectors on  $1$-unconditional bases could also be of interest for
other applications.

Given a standard class $\mathcal{S}$ and a space $X$, we consider the following indicator function:
$$
M_{X, \mathcal{S}}(n):=\sup \{\|x_1+\ldots+x_n\|: \; (x_j) \in \mathcal{S},\; \|x_j\|\leq 1\}.
$$

Lemma \ref{trivialestimate} can be rewritten as:

\begin{lemma}\label{estimativa} Let $\mathcal{S}$ be a standard class, and let $\Omega: X\lop X$
be an exact centralizer on a reflexive K\"othe function space.
If $\Omega$ is not disjointly singular, then there exists a subspace $W$ of $X$ generated by a
disjoint sequence and a constant $K$ such that given any $n$-tuple $(u_i)\in \mathcal{S}$ belonging
to the unit ball of $W$, one has
$$
\Big\|\Omega(\sum_{i=1}^n  u_i)-\sum_{i=1}^n \Omega(u_i)\Big\|
\leq K M_{X, \mathcal{S}}(n).
$$
\end{lemma}

We arrive now to the core of out method:

\begin{lemma}\label{core}
Let $(X_0, X_1)$ be an admissible couple of K\"othe function spaces, fix $0<\theta<1$, and let
$\Omega_\theta$ be the induced centralizer on $X_\theta$.
If $(x_i)\in \mathcal{S}$ is a $n$-tuple in the unit ball of ${X_\theta}$, then
\begin{equation}\label{eme}
\left\|\Omega_\theta\big(\sum_{i=1}^n x_i\big)-\sum_{i=1}^n \Omega_\theta(x_i)-
\log\frac{M_{X_0,\mathcal{S}}(n)}{M_{X_1,\mathcal{S}}(n)} \Big(\sum_{i=1}^n x_i\Big)\right\|\leq 3
\frac{M_{X_0, \mathcal{S}}(n)^{1-\theta} M_{X_1, \mathcal{S}}(n)^\theta}{\dist(\theta, \partial \mathbb S)}.
\end{equation}
\end{lemma}
\begin{proof}
To simplify, let us write $M(n,z) = M_{X_0, \mathcal{S}}(n)^{1-z} M_{X_1, \mathcal{S}}(n)^z$.
Given $0<\epsilon<1/4$, let $(x_i)\in \mathcal{S}$ be a $n$-tuple in the unit ball of $X_\theta$.
Let $B_\theta$ be a $(1+\epsilon)$-bounded selection $X_\theta \rightarrow {\mathcal H}$
such that ${\rm supp\ } B_\theta(x) \subset {\rm supp\ }x$ for all $x$.
Let $F_i=B_\theta(x_i)$ for each $i$.
Note that $\big(F_i(z)\big)$ is a $n$-tuple in $\mathcal{S}$ for any $z$ in the strip.
Let $F$ be the function
$$
F(z)=\frac{F_1(z)+\cdots+F_n(z)}{M(n,z)}
$$
for $z \in\mathbb{S}$. We know that  $\|F\| \leq 1+\epsilon$ and
$$
F(\theta)=\frac{1}{M(n,\theta)}(x_1+\ldots+x_n).
$$

Set $k=\|F(\theta)\|^{-1}.$ The map $\Phi: F(\theta)\mapsto F$ defines a linear
bounded selection on the one-dimensional subspace  $[F(\theta)]$ having norm at most $k$.
Therefore
$$
\|{B_\theta}_{|[F(\theta)]} - \Phi\| \leq 1+\epsilon + k \leq k(1+\epsilon)+\epsilon+k.
$$
Thus, if $x\in [F(\theta)]$, and denoting $\delta'=\delta_\theta^{\prime}$,
$$
\|(\delta'B_\theta - \delta'\Phi) (x)\|_\theta \leq (2k+k\epsilon+\epsilon)
\|\delta^{\prime}_{|{\rm ker}\ \delta_\theta}\| \|x\|_{\theta}.
$$

In particular
$$
\left\|(\delta'B_\theta - \delta'\Phi)(\sum_{i=1}^n x_i)\right \|_\theta\leq
(2k+k\epsilon+\epsilon)
\dist(\theta, \partial \mathbb S)^{-1} \left \|\sum_{i=1}^n x_i\right \|_{\theta},
$$
or equivalently
$$
\left\|(\delta'B_\theta - \delta'\Phi)(\sum_{i=1}^n x_i)\right \|_\theta
\leq \dist(\theta, \partial \mathbb S)^{-1} (2+\epsilon+\frac{\epsilon}{k}) M(n,\theta)
\leq 3\dist(\theta, \partial \mathbb S)^{-1} M(n,\theta).
$$
On the other hand,
$$
F'(\theta)=F(\theta) \log\frac{M_{X_0, \mathcal{S}}(n)}{M_{X_1, \mathcal{S}}(n)}+
\frac{1}{M(n,\theta)} \sum_i B_\theta(x_i)^{\prime}(\theta),
$$
which means
$$\delta'\Phi(\sum_i x_i)=\log\frac{M_{X_0, \mathcal{S}}(n)}{M_{X_1, \mathcal{S}}(n)}
\big(\sum_i x_i\big) +\sum_i \delta'B_\theta(x_i).$$
Therefore
$$\delta'\Phi(\sum_i x_i)- \delta'B_\theta(\sum_i x_i) = \sum_i \delta'B_\theta(x_i)-
\delta'B_\theta(\sum_i x_i)+\log\frac{M_{X_0, \mathcal{S}}(n)}{M_{X_1, \mathcal{S}}(n)}
\big(\sum_i x_i\big)$$
which yields
$$\Big\|\sum_i \delta'B_\theta(x_i)-\delta'B_\theta(\sum_i x_i)+
\log\frac{M_{X_0, \mathcal{S}}(n)}{M_{X_1, \mathcal{S}}(n)}
\big(\sum_i y_i\big) \Big\|_\theta \leq 3\dist(\theta, \partial \mathbb S)^{-1} M(n,\theta),$$
hence

\begin{equation}\label{newestimate}
\Big\|\Omega_\theta(\sum_{i=1}^n x_i)-\sum_{i=1}^n \Omega_\theta(x_i)-
\log\frac{M_{X_0, \mathcal{S}}(n)}{M_{X_1, \mathcal{S}}(n)} \big(\sum_{i=1}^n x_i\big)\Big\|
\leq  3\dist(\theta, \partial \mathbb S)^{-1} M(n,\theta)
\end{equation}
as desired.
\end{proof}

Observe that the estimate above applies (after suitable normalization) to all real centralizers; it is not equally clear the form such estimate should adopt for complex centralizers or for centralizers generated by arbitrary families. We show now that the function $\theta \mapsto M_{X_\theta, \mathcal{S}}(n)$ is log-convex:

\begin{lemma}\label{logconvex}
Given an interpolation scale $(X_\theta)$ of K\"othe function spaces associated to a pair
$(X_0,X_1)$ one has $$M_{X_\theta, \mathcal{S}}(n) \leq
M_{X_0, \mathcal{S}}(n)^{1-\theta} M_{X_1,\mathcal{S}}(n)^\theta.$$\end{lemma}

\begin{proof}
Let $F(z)=(F_1(z)+\cdots+F_n(z))/M(n, z)$ be the function in the proof of Lemma \ref{core}.
The inequalities
$\|F(\theta)\|_\theta \leq \|F\| \leq 1+\epsilon$ imply
$\|x_1+\cdots+x_n\|_\theta \leq (1+\epsilon) M(n, \theta)$, from where the conclusion follows.
\end{proof}

\section{Criteria for singularity}\label{sect:singular}

Here we give some results that will allow us to recognize nontrivial exact sequences
by showing that the quasi-linear map is singular in some sense.

\subsection{A general criterion in K\"othe function spaces}

We set now the core of our criterion to obtain disjointly singular sequences: to combine
Lemma \ref{estimativa}, Lemma \ref{core} and Lemma \ref{logconvex} to get the following
result.

\begin{prop}\label{general}
Let $\mathcal{S}$ be a standard class.
Let $(X_0, X_1)$ be an interpolation couple of K\"othe function spaces generating the
interpolation scale $(X_\theta)$; assume $X_\theta$ is reflexive and let $\Omega_\theta$ be
the induced centralizer on $X_\theta$, $0<\theta<1$.
If $\Omega_\theta$ is not disjointly singular then there exists a subspace $W \subset X_\theta$
spanned by disjointly supported vectors and a constant $K$ such that
\begin{equation}\label{lastMgeneral}
\left| \log \frac{M_{X_0,\mathcal{S}}(n)}{M_{X_1,\mathcal{S}}(n)}\right| M_{W,\mathcal{S}}(n)
\leq K M_{X_0,\mathcal{S}}(n)^{1-\theta}M_{X_1,\mathcal{S}}(n)^\theta.
\end{equation}
\end{prop}

An even more general criterion could be obtained by using in the definition of $M_X$
sequences of vectors whose norms are at most some prescribed varying values, instead of
vectors of norm at most $1$. We shall not write it since it will not be needed to deal with the applications we are
interested in.\\

We consider first the standard class $\mathcal D$ of all disjointly supported sequences
in a K\"othe function space $X$, and simplify notation to:
$$
M_{X}(n)= M_{X,\mathcal{D}}(n) = \sup \{\|x_1+\ldots+x_n\| :
x_1,\ldots,x_k\; \textrm{\rm disjoint in the unit ball of}\; X\}.
$$
Recall that two functions $f,g: \N\to \R$ are called equivalent, and denoted $f\sim g$,
if $0< \liminf f(n)/g(n) \leq \limsup f(n)/g(n)< +\infty$.
As a direct application of the criterion in Proposition \ref{general} we have:

\begin{prop}\label{good}
Let $(X_0, X_1)$ be an interpolation couple of two K\"othe function spaces so that
$M_{X_0}$ and $M_{X_1}$ are not equivalent. Let $0<\theta<1$.
Assume that $X_\theta$ is reflexive, "self-similar" in the sense that $M_W \sim M_{X_\theta}$ for every
infinite-dimensional subspace generated by a disjoint sequence $W\subset X_\theta$, and
$M_{X_{\theta}} \sim M_{X_0}^{1-\theta} M_{X_1}^\theta$.
Then $\Omega_\theta$ is disjointly singular.
\end{prop}

\begin{proof}
Otherwise, the estimate (\ref{lastMgeneral})  yields that, on some subspace $W$, one gets
$$
\left|\log\frac{M_{X_0}(n)}{M_{X_1}(n)}\right| M_{W}(n) = O( M(n,\theta))
= O(M_{X_\theta}(n)) = O(M_W(n)),
$$
which is impossible unless $M_{X_0}$ and $M_{X_1}$ are equivalent.
\end{proof}

Let us see these criteria at work.
The simplest case of course concerns the scale of $\ell_p$ spaces, $1< p<+\infty$.
%
These spaces are self similar with $M_{\ell_p}(n) = n^{1/p}$, while reiteration
theorems allow one to fix $X_0$ and $X_1$ at any two different values $p,q$ so that
$\lim |\log\frac{M_{X_0}(n)}{M_{X_1}(n)}| = \lim |\log n^{1/p - 1/q}| = +\infty$.
Thus, the induced centralizer, which is actually (projectively equivalent to) the Kalton-Peck
$\ell_\infty$-centralizer $\mathscr K$, is disjointly singular, hence singular on $\ell_p$.
The case of $L_p$ spaces, $1 < p < +\infty$ is also simple:
Proposition \ref{general} yields that if the twisted sum fails to be disjointly singular then
$$
\left|\log \frac{M_{L_\infty}(n)}{M_{L_1}(n)}\right| M_{\ell_p}(n)\leq
K M_{L_\infty}^{1-\frac{1}{p}}(n)M^{\frac{1}{p}}_{L_1}(n).
$$
Therefore $(\log n) n^{1/p} \leq K n^{1/p}$, which is impossible.
So the induced centralizer, which is actually (projectively equivalent to) the Kalton-Peck
$L_\infty$-centralizer $\mathscr K$, in $L_p$ is disjointly singular.\\

\begin{prop}\label{Lp} For $1<p<+\infty$, the Kalton-Peck $L_\infty$-centralizer
${\mathscr K}(f)=f \log \frac{|f|}{\|f\|}$ is disjointly singular on $L_p$,
\end{prop}

In \cite{cabe} it was shown that no $L_\infty$-centralizer on $L_p$ is singular for
$0<p<\infty$; previously, it had been shown in \cite{suarez} that the Kalton-Peck
$L_\infty$-centralizer $\Omega(f)=f \log |f|/\|f\|$ on $L_p$ is not singular since it becomes
trivial on the Rademacher copy of $\ell_2$. Proposition \ref{Lp} tells us that it is not trivial on any subspace generated by disjointly supported vectors. In \cite[Theorem 2(b)]{ccs} it was shown that the Kalton-Peck centralizer on $\ell_p$ is
singular for $0<p<\infty$. Cabello \cite{cabe} remarks that it would be interesting to know whether there exist singular
quasi-linear maps $L_p \to L_p$ for $p<2$.\\

A tricky question is what occurs with the scale of $L_p$-spaces in their $\ell_\infty$-module
structure generated by the Haar basis. Is the associated $\ell_\infty$-centralizer $\Omega_\theta$ singular? Khintchine's inequality makes possible to define $B_\theta(r)=f_r$ (the constant function
$f_r(z)=r$ on the subspace $\ell_2^R$ generated by the Rademacher functions, so
$\Omega_\theta(r)= \delta_\theta'B_\theta(r)=0$ on $\ell_2^R$ and thus $\Omega_\theta$
is not singular. Since the Haar basis is unconditional, this means that it is not disjointly singular either.

In sharp constrast with this, it was shown in \cite{ccs} that the Kalton-Peck centralizer $f\to f \log \frac{|f|}{\|f\|}$ (relative to the Haar basis)
is singular for $2\leq p<\infty$. This means, in particular, that the Kalton-Peck $\ell_\infty$-centralizer
relative to the Haar basis is not the $\ell_\infty$-centralizer induced by the interpolation
scale of $L_p$ spaces in their $\ell_\infty$-module structure.\\


We may use Proposition \ref{good} together with Proposition \ref{spirit} to prove singularity
of Kalton-Peck maps on more general classes of Banach lattices.

\begin{teor}\label{tsi}
Let $X$ be a reflexive, $p$-convex K\"othe function space, $p>1$.
Assume $M_X(n) \sim M_Y(n)$ for every subspace $Y$ of $X$ generated by a sequence of
disjointly supported vectors.
Then the Kalton-Peck map $\mathscr K(x)=x\log \frac{|x|}{\|x\|}$is disjointly singular on $X$.
\end{teor}
\begin{proof}
Since $X$ is $p$-convex we may by Proposition \ref{spirit} write $X=(L_\infty,X^{p})_{1/p}$.
Furthermore the centralizer induced by this interpolation scheme is a multiple of the
Kalton-Peck map.
In particular, the two twisted sums are projectively equivalent in the sense of
Section \ref{sect:twisted-sums}.
Thus one is singular if and only if the other is.
Since the norm on $X^{p}$ is defined as $\|x\|=\||x|^{1/p}\|_X^p$, we have immediately
that $M_{X^p}(n)=M_X(n)^p$.
Since $X$ is $p$-convex, $M_X(n)$ is not bounded and so $M_X(n)^p$ is not equivalent
to $M_{L_\infty}(n)=1$.
Furthermore
$$
M_{L_\infty}(n)^{1-\frac{1}{p}} M_{X^p}(n)^{\frac{1}{p}}
=(M_X(n)^p)^{1/p}=M_X(n),
$$
and by Proposition \ref{good} the centralizer (hence the Kalton-Peck map)
is disjointly  singular.
\end{proof}

\subsection{The criterion in spaces with  unconditional bases} We consider now the following asymptotic variation of $M_X$ with its associated standard class:
$$
A_{X}(n)=\sup \{\|x_1+\ldots+x_n\|: \|x_i\|\leq 1,\; n<x_1< \ldots < x_n\},
$$
when $X$ has a $1$-monotone basis.
Then Proposition \ref{general} can be reformulated as follows:

\begin{prop}\label{logb}
Let $(X_0, X_1)$ be an admissible pair of Banach spaces with a common $1$-unconditional
basis, and $0<\theta<1$.
\begin{itemize}
\item[a)] If the associated centralizer $\Omega_\theta$ is not singular then there
exists a block subspace $W \subset X_\theta$  and a constant $K$ such that:
$$\left| \log \frac{A_{X_0}(n)}{A_{X_1}(n)}\right| A_W(n) \leq K
{\rm A}_{X_0}^{1-\theta}(n){\rm A}_{X_1}^\theta(n).$$
\item[b)] If $A_{X_0} \not\sim A_{X_1}$ and
$A_{X_0}^{1-\theta} A_{X_1}^\theta \sim A_{X_{\theta}} \sim A_Y$
%
for all subspaces $Y\subset X_\theta$ then $\Omega_\theta$ is singular.
\end{itemize}
\end{prop}

Recall that a Banach space with a basis is said to be {\em asymptotically $\ell_p$} if there
exists $C \geq 1$ such that for all $n$ and normalized $n<x_1<\ldots<x_n$ in $X$, the sequence
$(x_i)_{i=1}^n$ is $C$-equivalent to the basis of $\ell_p^n$.
Apart from the $\ell_p$ spaces, Tsirelson's space is asymptotically $\ell_1$ as well as a class
of H.I. spaces (this one without unconditional basic sequences)  defined by Argyros and Delyanii \cite{argdely}.
One has:

\begin{cor}\label{gen-cor}
Let $(X_0, X_1)$ be an interpolation pair of Banach spaces with a common $1$-unconditional basis.
Let $p_0 \neq p_1$ and $\frac{1}{p}=\frac{1-\theta}{p_0}+\frac{\theta}{p_1}$.
The induced centralizer $\Omega_\theta: X_\theta \lop X_\theta$ is singular in any of the
following cases:

\begin{enumerate}
\item  The spaces $X_j$, $j=0,1$ are reflexive asymptotically $\ell_{p_j}$.
\item  Successive vectors in $X_{j}$, $j=0,1$ satisfy an asymptotic upper
$\ell_{p_j}$-estimate; and for every block-subspace $W$ of $X_\theta$, there exist a
constant $C$ and, for each $n$, a finite block-sequence $n < y_1 < \ldots < y_n$ in $B_W$
such that $\|y_1+\cdots+y_n\| \geq C^{-1}n^{1/p}$.
\end{enumerate}
\end{cor}

We also obtain as immediate corollary, with the same method as in Theorem \ref{tsi}:

\begin{cor}\label{bis}
Let $X$ be a $p$-convex reflexive space with $1$-unconditional basis, such that
$A_X(n) \sim A_Y(n)$ for every block-subspace $Y$ of $X$.
Then the Kalton-Peck map $\mathscr K(x)=x\log \frac{|x|}{\|x\|}$ is singular on $X$.
\end{cor}

Spaces to which Corollary \ref{bis} apply include, for example, the $p$-convexified Tsirelson
spaces $T^{(p)}$, $p>1$; since then $A_Y(n) \sim n^{1/p}$ for any block subspace $Y$.Thus, the Kalton-Peck map on $T^{(p)}$ is singular.

\subsection{The  criterion in spaces with monotone bases} Let $\Omega: X\to X$ be a quasi-linear map acting on a space with $1$-monotone basis.
This case does not fit under the umbrella of Kalton theorem, so it could well occur
that $\Omega$ could not be recovered from an interpolation scheme.
Without the lattice structure, supports cannot be used in the same way as before,
although successive vectors and asymptoticity still makes sense, so that the function
$A_X$ still may be defined.
In this context one uses the range of vectors (${\rm ran\ } x$ is the minimal interval
of integers containing its support) instead of their supports.
In the general case of $1$-monotone bases the maps $\Omega_\theta$ appearing in an
interpolation process are not $\ell_\infty$-centralizers or contractive.
However, the maps can be chosen to be ``range'' contractive,
in the sense of verifying ${\rm ran\ }\Omega_\theta(x) \subset {\rm ran}\ x$.
Indeed if for $x \in c_{00}$, $b_\theta(x)$ is an almost optimal selection, then
$B_\theta(x)=1_{{\rm ran} x}  b_{\theta}(x)$ will also be almost optimal and range
contractive, so $\delta_\theta^{\prime}B_\theta$ will be the required map.
The transfer principle still works and thus a non-singular $\Omega: X\to X$ must be
trivial on some subspace $W$ generated by blocks of the basis.\\

Note that the lattice structure was not used in the proof of Lemma \ref{core}, apart from
the use of supports, which are here replaced by ranges. So a proof entirely similar to that
of Lemma \ref{core},  using instead the function
$$
F(z) =\frac{1}{A_{X_0}(n)^{1-z} A_{X_1}(n)^z}(B_\theta(y_1) + \cdots + B_\theta(y_n))(z),
$$
immediately yields the estimate
\begin{equation}\label{acore}
\Big\|\Omega_\theta(\sum_{i=1}^n y_i)-\sum_{i=1}^n \Omega_\theta(y_i)-
\log\frac{{\rm A}_{X_0}(n)}{{\rm A}_{X_1}(n)}
\sum_i y_i\Big\| \leq k_\theta {\rm A}_{X_0}^{1-\theta}{\rm A}_{X_1}^\theta(n),
\end{equation}
for all $n<y_1<\cdots<y_n$ in the unit ball of $X_\theta$, in an interpolation scale $(X_0,X_1)$
of spaces with common $1$-monotone basis (here $k_\theta=3 \dist(\theta,\partial \mathbb S)^{-1}$).
One can also prove that the function
$\theta \mapsto A_{X_\theta}(n)$ is log-convex working as in Lemma \ref{logconvex}.
On the other hand the estimate in Lemma \ref{trivialestimate} requires lattice structure in a deep
way, and so something new is needed in the conditional case:
we shall now see how the lattice structure may be replaced by hypotheses of local unconditionality
and complementation.


\begin{prop}\label{appl-B-convex}
Assume we have a complex interpolation scheme of two spaces $X_0$, $X_1$ with a
common $1$-monotone basis.
Assume that for every block-subspace $W$ of $X_\theta$, there exists for every $n$ a finite
successive sequence $n<y_1<\cdots<y_n$ with $\|y_i\| \leq 1\ \forall i=1,\ldots,n$,
and constants $\varepsilon_n, \lambda_n, M_n$ satisfying
\begin{itemize}
\item[(i)] The block sequence is  $\varepsilon_n$-optimal, in the sense that
$\left\|\sum_{i=1}^n  y_i\right \| \geq \varepsilon_n A_{X_0}(n)^{1-\theta}A_{X_1}(n)^\theta$;
\item[(ii)] The block sequence $\{y_1, \dots, y_n\}$ is $\lambda_n$-unconditional;
\item[(iii)] the space $[y_1,\ldots,y_n]$ is $M_n$-complemented in $X_\theta$;
\end{itemize}
and so that
$$
\liminf_{n\to +\infty} \frac{\lambda_n^3 M_n}{\varepsilon_n \left|\log
\frac{A_{X_0}(n)}{A_{X_1}(n)}\right|}=0.
$$
Then $\Omega_\theta$ is singular.
\end{prop}
\begin{proof}
Suppose that the restriction of $\Omega_\theta$ to some subspace of $X$ is trivial.
By the hypothesis $\Omega_\theta$ is trivial on some block subspace $Y_\theta$
of $X_\theta$, and we can pick for any $n$ a $\lambda_n$-unconditional, $\varepsilon_n$-optimal,
finite sequence $[y_i]_{i=1}^n$ of blocks in $B_{Y_\theta}$ that is $M_n$-complemented in
$X_\theta$ by a projection $P_n$.

Then a local version of the proof of Lemma \ref{propersing} (3) can be made.
Let $\ell: {Y_\theta} \to L_0$ be a linear map so that
$\|\Omega_{|{Y_\theta}} - \ell\|\leq K$.
Let then $G_n \simeq \{-1,1\}^n$ be the group of units of $\ell_\infty^n$ acting on
$Y_n=[y_1,\ldots,y_n]$ in the natural way by change of signs of the coordinates on the $y_i$'s,
and let, for $y \in Y_n$, $\psi_n(y)$ be the finite average
$$
\psi_{n}(y) ={\rm Ave}_{u \in G_n}
u P_{n}(\Omega_{|{Y_\theta}}- \ell)(uy).
$$
Note that $\psi_n$ takes values in $Y_n$, and that this homogeneous map is bounded by
$K M_n \lambda_n^2$. It is also an exact $\ell_\infty^n$-centralizer in the sense
that $\psi_n(uy)=u\psi_n(y)$ for $u \in G_n$,
so $\mathrm{supp}\; \psi_{n}(y) \subset  \mathrm{supp}\; y$ for $y\in {Y_n}$.
This implies that $\psi_{n}(y_i)= \mu_i y_i$ for some scalars
$\mu_i$ with $|\mu_i|\leq KM_n\lambda_n^2$.
Thus
\begin{equation}\label{timo}
\begin{array}{ll}
\|\psi_{n}(\sum_{i=1}^n  y_i) - \sum_{i=1}^n  \psi_{n}(y_i)\|
  &\leq KM_n\lambda_n^2 \|\sum_{i=1}^n  y_i\| + \|\sum_{i=1}^n  \mu_i y_i\|\\
  &\leq KM\lambda_n^2 (1+\lambda_n) \|\sum_{i=1}^n y_i\|.
\end{array}
\end{equation}

Consider the estimate (\ref{acore}), and observe that replacing $\Omega_\theta$ by
$\Omega_\theta - \ell$ with $\ell$ linear changes nothing, and projecting
and averaging on $\pm$ signs as in the definition of $\psi_n$ only changes the estimate
by $\lambda_n \|P_n\| \leq \lambda_n M_n$; so one gets
$$
\Big\|\psi_n(\sum_{i=1}^n y_i)-\sum_{i=1}^n \psi_n(y_i)-
\log\frac{{\rm A}_{X_0}(n)}{{\rm A}_{X_1}(n)} \sum_{i=1}^n  y_i\Big\| \leq
k_\theta M_n \lambda_n A_{X_0}(n)^{1-\theta}A_{X_1}(n)^\theta.
$$
On the other hand by log-convexity of $A_{X_\theta}$ we can rewrite (\ref{timo}) as
 \begin{equation}\label{(1)}
\Big\|\psi_n(\sum_i  y_i)-\sum_i  \psi_n(y_i)\Big\|
\leq KM_n\lambda_n^2 (1+\lambda_n) {\rm A}_{X_0}^{1-\theta}(n){\rm A}_{X_1}^\theta(n).
\end{equation}

Putting both estimates together we get
$$
\Big|\log \frac{A_{X_0}(n)}{A_{X_1}(n)}\Big|\cdot \Big\|\sum_{i=1}^n y_i \Big\| \leq
\big(K\lambda_n(1+\lambda_n) +k_\theta\big) M_n \lambda_n {\rm A}_{X_0}^{1-\theta}(n){\rm A}_{X_1}^\theta(n).
$$
Condition (i)  yields that
$$
\varepsilon_n  \left|\log \frac{A_{X_0}(n)}{A_{X_1}(n)}\right| \leq
\big(K\lambda_n(1+\lambda_n)+k_\theta\big)M_n \lambda_n
$$
in contradiction with the hypothesis.
\end{proof}

\begin{cor}\label{generalmono}
Assume we have an interpolation scheme of two spaces $X_0$ and $X_1$ with a common
$1$-monotone basis. Let $1 \leq p_0 \neq p_1 \leq +\infty$, $0<\theta<1$, and
$\frac{1}{p}=\frac{1-\theta}{p_0}+\frac{\theta}{p_1}$ and assume that the spaces
$X_j$, $j=0,1$ satisfy an asymptotic upper $\ell_{p_j}$-estimate;
and that for every block-subspace $W$ of $X_\theta$, there exist a constant $C$
and for each $n$, a $C$-unconditional finite block-sequence $n < y_1 < \ldots < y_n$
in $B_W$ such that $\|y_1 + \cdots + y_n\| \geq C^{-1} n^{1/p}$
and $[y_1, \cdots, y_n]$ is  $C$-complemented in $X_\theta$. Then $\Omega_\theta$ is singular.
\end{cor}

It was proved by Pisier \cite{Pisier-Annals} that a B-convex Banach space contains
$\ell_2^n$ uniformly complemented.
Condition (ii) in Proposition \ref{appl-B-convex} could suggest to apply this result to
B-convex Banach spaces.
Proposition \ref{Ext-B-convex} below states that when $X$ is  B-convex, nontrivial twisted
sums $X\oplus_F X$ always exist.

\subsection{Interpolation of families of spaces}\label{interp-family}

Here we apply the preceding criteria to spaces induced by complex interpolation of a family
of spaces (see \cite{CCRSW2}), as will be necessary in Section \ref{sect:twisting-Ferenczi}. We thus take a family of compatible Banach spaces $\{X_{(j,t)} : j=0,1; t\in\R \}$ with index
in the boundary of $\mathbb S$, and denote by $\Sigma(X_{j,t})$ the algebraic sum of these
spaces with the norm
$$
\|x\|_\Sigma = \inf\{\|x_1\|_{(j_1,t_1)}+\cdots+\|x_n\|_{(j_n,t_n)} : x = x_1+\cdots+x_n\}.
$$

Let $\cl{H}(X_{j,t})$ denote the space of functions $g:\mathbb S\to \Sigma:=\Sigma(X_{j,t})$
which are $\|\cdot\|_\Sigma$-bounded, $\|\cdot\|_\Sigma$-continuous on $\mathbb S$ and
$\|\cdot\|_\Sigma$-analytic on $\mathbb S^\circ$; and satisfy $g(it)\in X_{(0,t)}$ and
$g(it+1)\in X_{(1,t)}$ for each $t\in\R$.
Note that $\cl{H}(X_{j,t})$ is a Banach space under the norm
$$
\|g\|_\mathcal H= \sup\{\|g(j+it)\|_{(j,t)}: j=0,1; t\in\R \}.
$$

For each $\theta\in (0,1)$, or even $\theta\in\mathbb S$, we define
$$
X_\theta :=\{x\in\Sigma(X_{j,t}) : x=g(\theta) \text{ for some } g\in\cl{H}(X_{j,t})\}
$$
with the norm $\|x\|_\theta=\inf\{\|g\|_\mathcal{H}: x=g(\theta)\}$.
Clearly $X_\theta$ is the quotient of $\cl{H}(X_{j,t})$ by the kernel of the evaluation map
$\ker\delta_\theta$, and thus it is a Banach space.

All the ingredients of our constructions straightforwardly adapt to this context, and the only
relevant modification is to set $A_j(n)={\rm ess\, sup}_{t \in \R}\; A_{X_{j+it}}(n)$ instead
of $A_{X_j}(n)$, $j=0,1$.

\begin{prop}\label{gen}
Consider an interpolation scheme given by a family $\{X_{(j,t)} : j=0,1; t\in\R \}$ of spaces
with a common $1$-monotone basis.
Let $1 \leq p_0 \neq p_1 \leq +\infty$, $0<\theta<1$, and $\frac{1}{p}=\frac{1-\theta}{p_0}+\frac{\theta}{p_1}$. Assume that all the spaces $X_{j,t}$ satisfy an asymptotic upper $\ell_{p_j}$-estimate with
uniform constant; and for every block-subspace $W$ of $X_\theta$, there exist a constant $C$
and for each $n$, a $C$-unconditional finite block-sequence $n < y_1 < \ldots < y_n$ in $B_W$
such that $\|y_1 + \cdots + y_n\| \geq C^{-1} n^{1/p}$ and $[y_1, \cdots, y_n]$ is
$C$-complemented in $X_\theta$. Then $\Omega_\theta$ is singular.
\end{prop}
\begin{proof}
It is similar to those of Proposition \ref{appl-B-convex} and Corollary \ref{generalmono}.
%
\end{proof}

\section{Singular twisted Hilbert spaces}
\label{sect:twisting-Hilbert}

In many cases, complex interpolation between a Banach space and its dual gives
$(X,X^*)_{1/2}=\ell_2$.
See e.g., the comments at \cite[around Theorem 3.1]{pisier-handbook}.
Also Watbled \cite{watbled} claims that her results cover the case of spaces
with a $1$-unconditional basis $X$.
We do not know whether there could be counterexamples with monotone basis.
So, for the sake of clarity, let us briefly explain the situation.

Given a Banach space $X$ with a normalized basis $(e_n)$, we denote by $(e^*_n)$ the
corresponding sequence of biorthogonal functionals.
We identify $X$ with $\{\big(e^*_n(x)\big) : x\in X)\}$, and its antidual space
$\hat X^*$ with $\{\overline{\big(x^*(e_n)\big)} : x^*\in X)\}$, both linear
subspaces of $\ell_\infty$, in such a way that $X\cap \hat X^*$ is continuously
embedded in $\ell_2$.
Indeed, $x=(a_n)\in X\cap \hat X^*$ implies
$x(x) =\sum |a_n|^2 \leq \|x\|_X \cdot \|x\|_{\hat X^*}$.

\begin{prop}\label{Wat-basis}
Let $X$ be a Banach space with a monotone shrinking basis.
Then $(X, \hat X^*)_{1/2}=\ell_2$ with equality of norms.
\end{prop}
\begin{proof}
It is enough to show that $\ell_2$ is continuously embedded in $X+ \hat X^*$ and
apply \cite[Corollary 4]{watbled}.
Let $T:X\cap \hat X^*\ra \ell_2$ be the embedding.
Since the basis is shrinking, $X\cap \hat X^*$ is dense in both $X$ and $\hat X^*$.
Thus the dual of $X\cap \hat X^*$ is $X^* + (\hat X^*)^* = X^{**}+\hat X^*$
\cite[2.7.1 Theorem]{Bergh-Lofstrom}, and the conjugate operator $T^*$ embeds $\ell_2$
into $X+\hat X^*$, which is a closed subspace of $X^{**}+\hat X^*$ by the arguments
in \cite[p. 204]{watbled}.
\end{proof}

We have a similar result for K\"othe function spaces $X$.
Observe that in this case $X^*$ and $\hat X^*$ coincide as sets.

\begin{prop}\label{Wat-Kothe} \emph{\cite[Corollary 5]{watbled}}
Let $X$ be a K\"othe function space on a complete $\sigma$-finite measurable space $S$.
Suppose that $X\cap X^*$ is dense in $X$ and
$$
L_1(S)\cap L_\infty(S) \subset X\cap X^* \subset L_2(S) \subset X+ X^*
\subset  L_1(S) + L_\infty(S).
$$
Then $(X, X^*)_{1/2}=L_2(S)$.
\end{prop}

Arguing like in Proposition \ref{Wat-basis}, we can show that the conditions $X$ and
$X^*$ intermediate spaces between $L_1(S)$ and $L_\infty(S)$, and $X\cap X^*$
dense in both $X$ and $X^*$ imply the hypothesis of Proposition \ref{Wat-Kothe}.
%

In all the previous situations the twisted sum space induced by the interpolation
of a space and its antidual is a twisted Hilbert space.
Proposition \ref{good} fits appropriately in this situation since $\ell_2$ is
``asymptotically self-similar'' in the sense that $A_W(n)=n^{1/2}$ for all infinite
dimensional block subspaces.
Thus, we are ready to construct singular exact sequences
$$
\begin{CD}0 @>>> \ell_2 @>>> E @>>> \ell_2 @>>>0.\end{CD}
$$
The first consequence of Corollary \ref{generalmono} is:

\begin{prop}\label{delya}
The interpolation of a reflexive asymptotically $\ell_p$ space, $p \neq 2$,
with its antidual induces a singular twisted Hilbert space.
\end{prop}

Thus inter\-polation of Tsirelson's space $\mathcal T$ with its dual $\mathcal T^*$;
or interpolation of Argyros-Deliyanni's H.I. asymptotically $\ell_1$-space
\cite{argdely} with its antidual produce new singular exact sequences
$$\begin{CD}
0@>>> \ell_2 @>>> X @>>> \ell_2@>>>0.
\end{CD}$$
By uniqueness in Kalton's theorem (Proposition \ref{kaltonunique}), the singular sequence
induced by interpolation of $\mathcal T$ with $\mathcal T^*$ is not boundedly equivalent to
$$\begin{CD}
0@>>> \ell_2 @>>> Z_2 @>>> \ell_2@>>>0.
\end{CD}$$

Thus, by Corollary \ref{prepara}, they cannot be even equivalent.
In favorable situations this can be improved to be non-permutatively projectively equivalent.
Indeed, given a reflexive Banach space $X$ with normalized subsymmetric basis $(e_n)$,
we denote as usual \cite{lindtzaf1}
$$
\lambda_X(n):=\Big\|\sum_{i=1}^n e_i\Big\|_{X}.
$$
Then $\lambda_{X^*}(n) \simeq n/\lambda_X(n)$ (see \cite[Proposition 3.a.6]{lindtzaf1}). One has

\begin{prop}
Let $\ell_M$ be the symmetric Orlicz space with function $M_\alpha(t)=e^{-t^{-\alpha}}, \alpha>0$.
The induced centralizers at $\ell_2 = (\ell_M, \ell_M^*)_{1/2}$ for different values of $\alpha$ are
not permutatively projectively equivalent.
\end{prop}
\begin{proof}
Let $X$ and $Y$ be reflexive spaces with normalized $1$-unconditional and $1$-subsymmetric
bases, and let $\Omega$ (resp. $\Psi$) be the induced centralizers at $\ell_2$ defined on terms of the
Lozanovskii decompositions associated to $(X, X^*)_{1/2}$ (resp. $(Y, Y^*)_{1/2}$). Then
$$
\big(\Omega-\mu \Psi\big) (x)= \Big(\log\frac{|a_0(x)|}{|a_1(x)|}-
\mu\log\frac{|a_0^{\prime}(x)|}{|a_1^{\prime}(x)|}\Big) x.
$$

Pick $x=\sum_{i=1}^n x_i  e_i$ with $x_i=1/\sqrt{n}$ and apply the above formula with
\medskip

\quad $|a_0(x)|=\lambda_X(n)^{-1} 1_{[1,n]}$,\quad $|a_1(x)|=\frac{\lambda_{X}(n)}{n} 1_{[1,n]}$,\quad and

\quad $|a_0^{\prime}(x)|=\lambda_{Y}(n)^{-1} 1_{[1,n]}$,\quad $|a_1^{\prime}(x)|=
\frac{\lambda_{Y}(n)}{n} 1_{[1,n]}$.
\medskip


If $\Omega-\mu \Psi$ is trivial then it is bounded by Corollary \ref{prepara}, so the function
$\log(n\lambda_{X}(n)^{-2})-\mu \log(n\lambda_{Y}(n)^{-2})$ on $\N$ is bounded, which implies that
the functions $n\lambda_{X}(n)^{-2}$ and $(n\lambda_{Y}(n)^{-2})^\mu$ are equivalent.
It is not difficult to check that that is impossible for different $\alpha, \beta > 0$ since the
choice of $M_\alpha$ in the statement yields $\lambda_{\ell_{M_\alpha}}(n) \simeq (\log n)^{1/\alpha}$.
Since the symmetric Orlicz spaces have symmetric bases, the corresponding induced centralizers are
not even permutatively projectively equivalent.
\end{proof}

We have found no specific criterion to show when twisted Hilbert sums induced by interpolation of
spaces with subsymmetric bases are singular.
Let us move our attention back to asymptotically $\ell_p$ spaces.

\begin{prop}
Let $X, Y$ be  spaces with asymptotically $\ell_p$ normalized $1$-unconditional  bases,
$1 \leq p \leq +\infty$.
Then the singular twisted Hilbert sums induced by  the interpolation couples $(X, X^*)$
and $(Y, Y^*)$ at $1/2$  are (permutatively) projectively equivalent if and only
if the bases of $X$ and $Y$ are  (permutatively) equivalent.
\end{prop}

\begin{proof}
The key is to show that projective equivalence actually implies equivalence, hence bounded equivalence;
which implies, by Kalton's result (Proposition \ref{kaltonunique}), that the bases of $X$ and $Y$
are equivalent.

Assume thus that the induced centralizers are $\lambda$-projectively equivalent.
By Lemma \ref{addit-prop} (3) and  Corollary \ref{prepara}
$$
\sum_i a_i^2 \Big(\log\frac{\mu_i}{\nu_i}-\lambda\log\frac{\mu_i^{\prime}}{\nu_i^{\prime}}\Big)^2
\leq K,$$
whenever $x=\sum_i a_i e_i$ in $\ell_2$ is normalized, and
$a_i^2=\nu_i \mu_i={\nu_i^{\prime}} {\mu_i^{\prime}}$ with
$$
1\leq \|\sum_i \nu_i e_i\|_{X}, \|\sum_i\mu_i e_i\|_{X^*}, \|\sum_i \nu_i^{\prime} e_i\|_{Y},
\|\sum_i \mu_i^{\prime} e_i\|_{Y^*} \leq c.
$$
Taking $x$ with support far enough on the basis, we may choose $a_i=n^{-1/2}$ and
$\nu_i = \nu^{\prime}_i \simeq n^{-1/p}$, $\mu_i=\mu_i^{\prime} \simeq n^{-1/p'}$.
Then $|(1-\lambda)\log n|^2 \leq K'$, which means that $\lambda=1$.
Therefore we have equivalence, and even bounded equivalence by Corollary \ref{prepara}.

To deduce the permutative projective equivalence case from the projective equivalence case just
note that if a basis $(e_n)$ is asymptotically $\ell_p$ then any permutation of $(e_n)$ is again
asymptotically $\ell_p$ ``in the long distance", in the sense that there exists $C \geq 1$
and a function $f:\N\to \N$ such that for all $n$ and normalized $f(n)<x_1<\ldots<x_n$ in $X$,
the sequence $(x_i)_{i=1}^n$ is $C$-equivalent to the basis of $\ell_p^n$.
\end{proof}

From the purely Banach space theory it is interesting to decide whether the twisted Hilbert spaces
thus obtained are isomorphic. We can obtain non-isomorphic singular twisted Hilbert spaces as follows.

\adef
A Lipschitz function $\phi: [0+\infty) \rightarrow \mathbb{C}$ with $\phi(0)=0$ is called \emph{expansive}
if for every $M$ there exists $N$ such that $|s-t| \geq N \Rightarrow |\phi(s)-\phi(t)|\geq M$.
\zdef

Observe that Lipschitz functions for which $\lim_{t \to \infty} \phi'(t)=0$ are not expansive.
In particular the functions $\phi_r$ for $0<r<1$ are not expansive, while $\phi_1$ is expansive.

\begin{prop}\label{sing-lipsch}
Let $X$ be a  space with a normalized $1$-unconditional basis that is self-similar, in the sense that
$M_X \sim M_Y$ for all subspaces $Y \subset X$ generated by a disjoint sequence, and such that
$\lim_{n\ra\infty} M_X(n)=\infty$.
Assume $\phi: [0+\infty) \rightarrow \mathbb{C}$ is an expansive Lipschitz function. Then the Kalton-Peck map $\mathscr K_\phi(x)=x \phi\left (- \log \frac{|x|}{\|x\|}\right)$ is  singular.
\end{prop}
\begin{proof}
To simplify notation we write $\Omega=\mathscr K_\phi$.
Observe that $\Omega$ is a contractive centralizer. Assume that $Y$ is a sublattice of $X$ such that
$\Omega_{|Y}$ is trivial.
Let $M$ be arbitrary positive, $N$ be such that $|s-t|\geq N\Rightarrow |\phi(s)-\phi(t)|\geq M$,
and $n$ be such that $M_Y(n) \geq 2 e^N$.
We may consider disjoint vectors $y_1,\ldots,y_n$ in $Y$ of norm at most $1$ such that
$\|y_1+\cdots+y_n\| \geq M_Y(n)/2$.
An easy calculation shows that
$$
\Omega(\sum_i y_i)-\sum_i \Omega (y_i)=\sum_i y_i  \big(\phi( - \log(\sum_i y_i/K))-
\phi(- \log(\sum_i y_i))\big),
$$
where $K=\|\sum_{i=1}^n y_i\|$.
Each coordinate of the vector $\log(\sum_i y_i))-\log(\sum_i y_i/K)$ is  $\log K$ which
is larger than $\log (M_Y(n)/2) \geq N$.
Therefore each coordinate of the vector $\phi( -\log(\sum_i y_i))-\phi( -\log(\sum_i y_i/K))$
is larger than $M$ in modulus.
We deduce that
$$
\|\Omega(\sum_i y_i)-\sum_i \Omega (y_i) \| \geq M \|\sum_i y_i\| \geq M M_Y(n)/2.
$$
By Lemma \ref{estimativa}, this implies for some fixed constant $k$ that $kM_X(n) \geq MM_Y(n)/2$,
therefore $M_X \not\sim M_Y$, a contradiction which proves that $\Omega$ is singular.
\end{proof}

Observe that the condition $\lim_{n\ra \infty} M_X(n)=\infty$ can be obtained assuming that $X$ is self-similar
and does not contain $c_0$.\\

In \cite{kalt-canad} Kalton obtained a family $Z_2(\alpha)$ of complex twisted Hilbert spaces
induced by the centralizers
$$
\mathscr K_{i\alpha}(x) = x \left( -\log \frac{|x|}{\|x\|} \right)^{1+i\alpha}
$$
for $-\infty<\alpha<\infty$ (see also \cite{kaltmem}). Since these are not real centralizers they appear, according to \cite{kaltdiff}, as induced by the interpolation of three spaces. They are singular because:

\begin{lemma} The Lispchitz function $\phi(t)=t^{1+i\alpha}$ is expansive.
\end{lemma} \begin{proof} $
|\phi(s)-\phi(t)|=|se^{i\alpha \log(s)}-te^{i\alpha \log(t)}| \geq ||s|-|t||=|s-t|.
$\end{proof}

Thus, according to Proposition \ref{sing-lipsch} \cite{kalt-canad} we get:

\begin{prop}\label{newTwistHsing} Given $\alpha \in \R$, the exact sequences
$$\begin{CD}
0@>>> \ell_2 @>>> Z_2(\alpha) @>>> \ell_2 @>>> 0\end{CD}$$
are singular and for $\alpha \neq \beta$ the spaces
$Z_2(\alpha)$ and $Z_2(\beta)$ are not isomorphic.
\end{prop}

We consider now the Kalton-Peck centralizers $\mathscr{K}_{\phi_r}(x)=x \phi_r\big(-\log(|x|/\|x\|_2)\big)$
induced by the Lipschitz functions $\phi_r(t)= t$ for $0\leq t\leq 1$, and
$\phi_r(t)= t^r$ for $1<t<\infty$, and the twisted Hilbert spaces
$\ell_2(\phi_r) = \ell_2 \oplus_{\mathscr{K}_{\phi_r}} \ell_2$ they generate, introduced by Kalton and Peck in \cite{kaltpeck}.
Note that $\ell_2(\phi_1)=Z_2$. It follows from Kalton's theorem \ref{kalton} (\cite[Theorem 7.6]{kaltdiff}) that
$\ell_2(\phi_r)$ comes generated by some interpolation scale, and we show now that it is a scale of Orlicz spaces.

\begin{prop}\label{TwistHinterp}
Let $0<r<1$ and $\varphi_0, \ \varphi_1$ be the maps $[0,\infty) \rightarrow [0,\infty)$
defined by
$$
\varphi_0^{-1}(t)=t^{\frac{1}{2}+\frac{1}{4}(-\log t)^{r-1}}, \quad
\varphi_1^{-1}(t)=t^{\frac{1}{2}-\frac{1}{4}(-\log t)^{r-1}},
$$
on a neighborhood of $0$, and extended to $[0,\infty)$ to be  $N$-functions with the
$\Delta_2$-property. Then $$\ell_2(\phi_r) \simeq  (\ell_{\varphi_0},\ell_{\varphi_1})_{1/2}.$$
\end{prop}
\begin{proof}
We note that everything here is well defined since by choice of $r$ and after
an easy calculation,
$t^{3/4} \leq \varphi_0^{-1}(t) \leq t^{1/4}$, $t^{3/4} \leq \varphi_1^{-1}(t) \leq t^{1/4}$
and $\varphi_1^{-1}(t)$ and $\varphi_0^{-1}(t)$ are increasing, for $t$ in some neighborhood
of $0$.
This is enough to make sure that $\varphi_1$ and $\varphi_0$ define $N$-function Orlicz spaces.
The $\Delta_2$-property is also satisfied on a neighborhood of $0$.
Indeed
\begin{displaymath}
\begin{array}{ll}
\varphi_0^{-1}(9t)
&=3 t^{\frac{1}{2}+\frac{1}{4}(-\log 9 t)^{r-1}}
=3 \varphi_0^{-1}(t) t^{\frac{1}{4}[(-\log 9 -\log t)^{r-1}-(-\log t)^{r-1}]}\\
&=3 \varphi_0^{-1}(t) \exp \big(-\frac{1}{4} (-\log t)^r
[(1+\frac{\log 9}{\log t})^{r-1}-1]\big).
\end{array}
\end{displaymath}
The exponential in this expression is easily seen to tend to $1$ when $t$ tends to $0$,
so close enough to $0$, $\varphi_0^{-1}(9t) \geq 2\varphi_0^{-1}(t)$, and $\varphi_0$ satisfies
the $\Delta_2$ condition $\varphi_0(2s) \leq 9\varphi(s)$ for $s$ in a neighborhood of $0$.
The same holds for $\varphi_1$. Since $\varphi_0^{-1}(t)\varphi_1^{-1}(t) =t$ on a neighborhood of $0$,
the equality $(\ell_{\varphi_0},\ell_{\varphi_1})_{1/2}=\ell_2$ holds up to equivalence of bases.

Let $\psi$ be the map so that
$$
\varphi_1^{-1}(t)=t^{\frac{1}{2}-\frac{1}{4}\psi(-\log(t))}.
$$
Note that $\psi$ is continuous, $\psi(s)=s^{r-1}$ for $s$ on a neighborhood
$V$ of $+\infty$, and only the value of $\psi(s)$ for $s \geq 0$ is relevant here.
Suppose that $\|x\|_2=1$.
Then the centralizer $\Omega$ associated to $(\ell_{\varphi_0},\ell_{\varphi_1})_{1/2}=\ell_2$
(see Proposition \ref{interp-Orlicz-central}), is given by
$$
\Omega(x)=2x \log \frac{\varphi_1^{-1}(|x|^2)}{|x|}
=2x \log{|x|^{-\frac{1}{2}\psi(-\log |x|)}}=x \psi(-\log |x|) (-\log |x|),
$$
while $\mathscr{K}_{\phi_r}(x)_n=x_n\cdot (-\log|x_n|)^r$ whenever $|x_n|$ is less than some
constant $c$ depending on $V$.
So we deduce that
\begin{displaymath}
\begin{array}{ll}
\|\Omega(x)-\mathscr{K}_{\phi_r}(x)\|^2
&\leq \sum_{|x_n| \geq c} 2(\Omega(x))_n^2+(\mathscr{K}_{\phi_r}(x))_n^2\\
&\leq 2\big( (-\log c)^2 \sup_{[0,-\log c]}|\psi|+ (-\log c)^{2r}\big).
\end{array}
\end{displaymath}
Since $\Omega$ and $\mathscr{K}_{\phi_r}$ are homogeneous, they are boundedly equivalent.
Hence $\ell_2\oplus_\Omega\ell_2$ and $\ell_2(\phi_r)$ are isomorphic.
\end{proof}

Recall from \cite[Corollary 5.5]{kaltpeck} that the spaces $\ell_2(\phi_r)$ are mutually
non-isomorphic for different values of $0<r\leq1$.
We already know \cite[Corollary 5.5]{kaltpeck} that $\mathscr K = \mathscr K_{\phi_1}$ is
singular but, since the function $\phi_r$ is not expansive for $r<1$, we do not know if also
$\mathscr{K}_{\phi_r}$ is singular for $0<r<1$.

\section{The twisting of H.I. spaces}
\label{sect:twisting-H.I.}

A Banach space $X$ is said to be \emph{indecomposable} if it cannot be decomposed
as $A\oplus B$ for two infinite dimensional subspaces $A,B$.
An infinite dimensional space $X$ is said to be \emph{hereditarily indecomposable} (H.I., in short)
if all subspaces are indecomposable \cite{GM}.
It is said to be \emph{quotient hereditarily indecomposable} (Q.H.I., in short)
if all its quotients of subspaces are indecomposable \cite{FQHI}.
In particular, Q.H.I. spaces are H.I. The existence of Q.H.I. Banach spaces was proved in \cite{FQHI}.
The simplest connection between H.I. spaces and the theory of singular exact sequences is described
in the following folklore proposition; we present its proof for the sake of completeness.

\begin{lemma}\label{characterization}
Given an exact sequence of Banach spaces
$$
\begin{CD} 0@>>>Y@>>> X@>q>> Z@>>>0, \end{CD}
$$
the space $X$ is H.I. if and only if $Y$ is H.I. and $q$ is strictly singular.
\end{lemma}
\begin{proof}
Suppose $X$ is H.I. Then clearly $Y$ is H.I., and if $q$ is not strictly singular,
$q_{|V}$ is an isomorphism for some (infinite dimensional) subspace $V$ of $X$, hence
$Y \oplus V$ is a subspace of $X$ and thus $X$  cannot be H.I.
Conversely, suppose that $q$ is strictly singular.
If $X$ is not H.I. we can find a decomposable subspace $X_1\oplus X_2$ of $X$,
and $q$ has compact (even nuclear) restrictions on some subspaces $Y_1\subset X_1$ and
$Y_2\subset X_2$.
Thus we can assume that there exists a bijective isomorphism $U:X\ra X$
such that $U(Y_1)$ and $U(Y_2)$ are contained in $Y$.
Since $U(Y_1)\oplus U(Y_2)$ is closed, we conclude that $Y$ is not H.I.
%
\end{proof}

The basic question we tackle in this section is whether it is possible to obtain nontrivial
twisted sums of H.I. spaces.
The existence of a nontrivial twisted sum of $A$ and $B$ will be denoted $\Ext(B,A)\neq 0$.
On one hand, if $X$ is a given example of a Q.H.I. space and $Y$ is a subspace of $X$ with
$\dim Y = \dim X/Y =\infty$, then $X$ is a nontrivial twisted sum of the two H.I. spaces $Y$
and $X/Y$.
However, what one is looking for is to obtain methods to twist two specified H.I. spaces.
Recall that the Kalton-Peck method \cite{kaltpeck} to twist spaces only works, in principle,
under unconditionality assumptions.
A second method is to use the local theory of exact sequences as developed in \cite{cabecastuni}.
The following result is a good example; we could not find it explicitly in the literature,
but it is certainly known:

\begin{prop}\label{Ext-B-convex}
If $X$ is a B-convex Banach space then $\Ext(X,X)\neq 0$.
\end{prop}
\begin{proof}
If $X$ contains $\ell_2^n$ uniformly complemented, as it is the case of $B$-convex
Banach spaces, then $\Ext(X, \ell_2) \neq 0$ \cite{cabecastuni}.
And if $\Ext(X,X)=0$ then $\Ext(X,\ell_2)=0$ \cite{cabecastuni}.
\end{proof}

The only currently known $B$-convex H.I. space is the one constructed by Ferenczi
in \cite{fere}. So, calling this space $\mathcal F$ one gets $\Ext(\mathcal F, \mathcal F)\neq 0$.
However this is not entirely satisfactory since this twisting does not provide any information
about the twisted sum space, apart from its existence.
So we formulate the following question:

\begin{probl}\label{P2} Let $X$ be an H.I. space. Does there exist an H.I. twisted sum of $X$?
\end{probl}

Focusing again on Ferenczi's space $\mathcal F$, since it is a space obtained via an
interpolation scheme, i.e., $\mathcal F = X_\theta$ for a certain configuration of spaces,
the induced centralizer $\Omega_\theta$ provides a natural twisted sum of $\mathcal F$ with
itself that we call $\mathcal F_2$:
$$
\begin{CD}0@>>> \mathcal F  @>>> \mathcal F_2 @>>> \mathcal F @>>>0.
\end{CD}
$$

We will show in Section \ref{sect:twisting-Ferenczi} that this sequence is singular, which
implies that $\mathcal F_2$ is H.I.\\

By the characterization in Lemma \ref{characterization} it is tempting to believe
that a twisted sum of two H.I. spaces is H.I. whenever is not trivial.
However, this is not the case:

\begin{prop}\label{indecomposable}
There exists a nontrivial twisted sum of two H.I. spaces which is indecomposable
but not H.I.
\end{prop}

\begin{proof}
Recall that two Banach spaces $A,B$ are said to be totally incomparable if no infinite
dimensional subspace of $A$ is isomorphic to a subspace of $B$.
It was proved in \cite[Prop. 25]{FQHI} that there exist two reflexive Q.H.I. spaces
$X_1, X_2$ admitting infinite dimensional subspaces $Y_1\subset X_1$ and $Y_2\subset X_2$
such that $Y_1$ is isometric to $Y_2$ and $X_1/Y_1$ and $X_2/Y_2$
are infinite dimensional and totally incomparable.
Note that $X_1^*$ and $X_2^*$ are Q.H.I.

Given a bijective isometry $U:Y_1\ra Y_2$, we consider the subspace
$\hat Y :=\{(y,Uy): y\in Y_1\}$ of $X_1\times X_2$, the quotient
$\hat X := (X_1\times X_2)/\hat Y$, and the quotient map $Q:X_1\times X_2\ra \hat X$.
Note that $\hat X_1 := Q(X_1\times \{0\})$ and $\hat X_1 := Q(\{0\}\times X_2)$ are
subspaces of $\hat X$ isometric to $X_1$ and $X_2$ respectively, and
$\hat Z :=\hat X_1\cap \hat X_2 = Q(Y_1\times \{0\})=  Q(\{0\}\times Y_2)$.
Thus $\hat X/\hat Z$ is isomorphic to $\hat X_1/\hat Z\oplus \hat X_2/\hat Z$,
hence $\hat Z^\perp$ is decomposable and $\hat X^*$ is not H.I.
Let us see that $\hat X^*$ is a nontrivial twisted sum of two H.I. spaces:
Since $\hat X$ is reflexive and H.I. \cite[Proposition 23]{FQHI}, the dual space
$\hat X^*$ is indecomposable, hence the exact sequence
$$
\begin{CD}
0 @>>> \hat X_1^\perp @>>>\hat X^*@>>> \hat X^*/\hat X_1^\perp @>>> 0
\end{CD}
$$
is nontrivial.
Moreover, $\hat X_1^\perp$ and $\hat X^*/\hat X_1^\perp$ are H.I. because
$\hat X_1 \simeq X_1$ and $\hat X/\hat X_1 \simeq X_2/Y_2$ are Q.H.I. and reflexive.
\end{proof}

We can present  an alternative construction of nontrivial and non H.I. twisted sums
of H.I. spaces.
Let us say that a Banach space $X$ \emph{admits a singular extension} if there exists
a singular exact sequence
$$\begin{CD} 0@>>>X @>>> Y@>q>> Z@>>>0;\end{CD}$$
i.e., an exact sequence with $q$ strictly singular and $Z$ infinite dimensional.
By Lemma \ref{characterization}, a HI space admits a singular extension if and only if it
admits a non trivial extension which is a HI space.


\begin{prop}
Every separable H.I. space $X$ which admits a singular extension is a complemented subspace
of a nontrivial twisted sum of two H.I. spaces.
\end{prop}

\begin{proof}
Let $0 \ra X\stackrel{i}{\lra}Y\stackrel{q}{\lra}Z\ra 0$ be a singular extension of $X$
with $Y$ separable.
It follows from Proposition \ref{characterization} that $Y$ is H.I.
By \cite[Theorems 14.5 and 14.8]{AT} there exists a separable H.I. space $W$ and a surjective
operator $p:W\ra Y$ with infinite dimensional kernel.
Note that $p$ is strictly singular by Proposition \ref{characterization}.
We consider the closed subspace $\PB := \{(w,x) \in W\oplus X : p(w) = i(x)\}$ of $W\oplus X$
and the projection operators $\alpha:\PB\ra W$ and $\beta:\PB\ra X$.
Note that $\beta$ is strictly singular because $i\beta = q\alpha$, and that $\beta$ is
surjective with $\ker(\beta)= \ker(p)$ an H.I. space.
Hence $\PB$ is H.I.

Since the operator $U:(w,x)\in Z\oplus X \lra i(x)-p(w)\in Y$ is surjective,
we have a twisted sum of two H.I. spaces
\begin{equation}\label{decomposable}\begin{CD}
0 @>>> \PB @>>> W\oplus X  @>U>> Y @>>> 0.\end{CD}
\end{equation}
To finish the proof it is enough to show that this twisted sum is nontrivial.
Indeed, otherwise $U$ would be in the class $\Phi_r$ of operators with complemented kernel
and finite codimensional closed range.
By the stability of $\Phi_r$ under strictly singular perturbations \cite[Theorem 7.23]{Aiena},
the operator $T(w,x)\in Z\oplus W \lra i(x)\in Y$ would define an isomorphism of $X$ onto a
finite codimensional subspace of $Y$, which is not possible.
\end{proof}

We do not know if every separable H.I. space admit a singular extension. On the other hand, the exact sequence (\ref{decomposable}) also shows that there are nontrivial twisted sums of
H.I. spaces which are decomposable (``two" is the maximum number of summands by
\cite[Theorem 1]{gonzherre}). In Section \ref{sect:higher-twisting} we will give other examples of this kind.\\

To conclude this section, we formulate the general problem about twisting H.I:

\begin{probl}
Does there exists an H.I. space $X$ so that $\Ext(X,X)=0$?
\end{probl}

Note (see \cite{accgm}) that there are only a few known solutions to the equation $\Ext(X,X)=0$: the spaces
$L_1(\mu)$, $c_0$, $\ell_\infty(\Gamma)$ and $\ell_\infty/c_0$.

\section{An H.I. twisted sum of $\mathcal F$}\label{sect:twisting-Ferenczi}

Ferenczi's H.I. uniformly convex space $\mathcal F$ \cite{fere} comes induced by a
complex interpolation scheme associated to a family of Banach spaces (briefly described in
Subsection \ref{interp-family}) setting $X_{(1,t)} = \ell_q$, $q>1, t\in\R$, and as $X_{(0,t)}$ certain
Gowers-Maurey-like spaces with $1$-monotone basis. We fix $\theta\in (0,1)$, and define $
\mathcal{F} =\{x\in\Sigma(X_{j,t}) : x=g(\theta) \text{ for some } g\in\cl{H}(X_{j,t})\}
$ with the quotient norm of $\cl{H}(X_{j,t})/\ker\delta_\theta$, given by
$\|x\|_\theta=\inf\{\|g\|_\mathcal{H}: x=g(\theta)\}$. In this section we will show that the space $\mathcal F$ satisfies the hypotheses of Proposition \ref{gen} with $C=1+\epsilon$ for any $\epsilon>0$ and thus:

\begin{teor}\label{ferenczi} The induced exact sequence$$
\begin{CD}0@>>> \mathcal F  @>>> \mathcal F_2 @>>> \mathcal F @>>>0.
\end{CD}
$$is singular. Therefore $\mathcal F_2$ is H.I.
\end{teor}

We have trivial upper $\ell_1$-estimates in spaces $X_{(0,t)}$ and
upper $\ell_q$-estimates in spaces $X_{(1,t)}$.
So we only need to check the $\ell_p$-condition of Proposition \ref{gen} in the middle
space $X_\theta$, for $\frac{1}{p}=1-\theta+\frac{\theta}{q}$.
Let $f(x) :=\log_2(1+x)$.
We first state  estimates relative to successive vectors in the space $\mathcal F$
\cite[Proposition 1]{fere}, as well as estimates for successive functionals in $\mathcal F^*$
obtained by standard duality arguments:

\begin{lemma}\label{estimate}
For all successive vectors $x_1<\cdots<x_n$ in $\mathcal F$,
$$
\frac{1}{f(n)^{1-\theta}}\Big(\sum_{i=1}^n \|x_i\|^p\Big)^{1/p} \leq
\Big\|\sum_{i=1}^n x_i\Big\| \leq \Big(\sum_{i=1}^n \|x_i\|^p\Big)^{1/p},
$$
and for all successive functionals $\phi_1<\cdots<\phi_n$ in $\mathcal F^*$,
$$
\Big(\sum_{i=1}^n \|\phi_i\|^\pp\Big)^{1/\pp} \leq \Big\|\sum_{i=1}^n \phi_i\Big\|
\leq f(n)^{1-\theta}\Big(\sum_{i=1}^n \|\phi_i\|^\pp\Big)^{1/\pp}.
$$
\end{lemma}


In \cite{fere}, $\ell^n_{p+}$-averages are defined as normalized vectors of the form
$\sum_{i=1}^n x_i$, where the $x_i$'s are successive of norm at most $(1+\epsilon)n^{-1/p}$,
and may be found in any block-subspace of $\mathcal F$ (see \cite[Lemma 2]{fere}).
However here we need to control not only the norm of $\sum_{i=1}^n x_i$ but also of
$\sum_{i=1}^n \pm x_i$ for any choice of signs $\pm$, so \cite[Lemma 2]{fere} is not quite
enough.
To this end we shall use RIS sequences as defined in \cite[Definition 3]{fere}.

RIS sequences with constant $C>1$ are successive sequences of $\ell_{p+}^{n_k}$-averages
with a technical "rapidly" increasing condition on the $n_k$'s and therefore are also
present in every block subspace of $\mathcal F$.
Every subsequence of a RIS sequence is again a RIS sequence.
In what follows $L$ is some lacunary infinite subset of $\N$ whose exact
definition may be found in \cite{fere}.
As a consequence of Lemma \ref{estimate}, \cite[Lemma 10]{fere} and standard duality
arguments we have:

\begin{lemma}\label{RIS}
Let $y_1<\cdots<y_n$ be a RIS sequence in $\mathcal F$, with constant $1+\epsilon^2/100$,
where $n \in [\log N, {\rm exp}\ N]$ for some $N$ in $L$, and $0<\epsilon<1/16$.
Then
$$
\frac{n^{1/p}}{f(n)^{1-\theta}} \leq\|\sum_{i=1}^n  y_i\| \leq
(1+\epsilon) \frac{n^{1/p}}{f(n)^{1-\theta}}.
$$
Furthermore if for all $i$, $\phi_i\in \mathcal F^*$ satisfies
$\|\phi_i\|=\phi_i(y_i)=1$ and ${\rm ran\ }\phi_i \subset {\rm ran\ }y_i$, then
$$
(1+\epsilon)^{-1} f(n)^{1-\theta}n^{1/p'} \leq \|\sum_{i=1}^n  \phi_i\| \leq
f(n)^{1-\theta}n^{1/p'}.
$$
\end{lemma}


We deduce the existence of sequences satisfying the condition of Proposition \ref{gen}
in any block-subspace of $\mathcal F$:

\begin{prop}\label{seq}
Let $Y$ be a block sequence of $\mathcal F$, $n \in \N$, and $\epsilon>0$.
Then there exists a  block-sequence $y_1<\cdots<y_n$ in $Y$ and a  block-sequence
$\psi_1<\cdots<\psi_n$ in $\mathcal F^*$ such that:
\begin{itemize}
\item[(1)] $(1+\epsilon)^{-1} \leq \|\psi_i\| \leq 1 \leq \|y_j\| \leq 1+\epsilon$ and
$\psi_i(y_j)=\delta_{ij}$ for $i,j=1,\ldots,n$,
\item[(2)] for any complex $\alpha_1,\ldots,\alpha_n$,
$\|\sum_{i=1}^n \alpha_i y_i\| \geq (1+\epsilon)^{-1}(\sum_{i=1}^n|\alpha_i|^p)^{1/p}$
\item[(3)] for any complex $\alpha_1,\ldots,\alpha_n$,
$\|\sum_{i=1}^n \alpha_i \psi_i\| \leq (1+\epsilon)(\sum_{i=1}^n |\alpha_i|^{\pp})^{1/\pp}$
\end{itemize}
In particular the block sequence $y_1<\cdots<y_n$ of $Y$ is $(1+\epsilon)$-equivalent to the
unit vector basis of $\ell_p^n$ and $[y_1,\ldots,y_n]$ is $(1+\epsilon)$-complemented in $Y$.
\end{prop}

\begin{proof}
Assuming $\epsilon \leq 1/16$, pick $m$ such that $\dist(mn,N) <n$ for some $N \in L$
and big enough to ensure that $m$ and $mn$ belong to $[\log N,\exp N]$, and that
$f(mn)/f(m)<1+\epsilon$.
Denote $M=mn$.
Let $x_1,\ldots,x_M$ be a RIS in $Y$ with constant $1+\epsilon^2/100$ and
$\phi_1,\ldots,\phi_M$ be a sequence of successive norming functionals in $X^*$
for $x_1,\ldots,x_M$.

Now for $j=1,\ldots,n$, let
$$
y_j=\frac{f(m)^{1-\theta}}{m^{1/p}}\sum_{i=(j-1)m+1}^{jm} x_i,\, \textrm{ and }\,
\psi_j=\frac{1}{f(m)^{1-\theta} m^{1/\pp}}\sum_{i=(j-1)m+1}^{jm} \phi_j.
$$
Since $x_{(j-1)m+1},\ldots,x_{jm}$ is a RIS with constant $1+\epsilon^2/100$,
we have by Lemma \ref{RIS} that for $j=1,\ldots,n$,
$$
1 \leq \|y_j\| \leq (1+\epsilon),\quad  (1+\epsilon)^{-1} \leq \|\psi_j\| \leq 1,
$$
and clearly $\psi_j(y_k)=\delta_{j,k}.$
For any complex  $\alpha_1,\ldots,\alpha_n$,  Lemma \ref{estimate} implies
$$
\frac{m^{1/p}}{f(m)^{1-\theta}} \|\sum_{j=1}^n \alpha_j y_j\| \geq
\frac{(\sum_{j=1}^n m|\alpha_j|^p)^{1/p}}{f(M)^{1-\theta}},
$$
so
$$
\|\sum_{j=1}^n \alpha_j y_j\| \geq (\sum_{j=1}^n|\alpha_j|^p)^{1/p}
(\frac{f(m)}{f(M)})^{1-\theta} \geq (\sum_{j=1}^n|\alpha_j|^p)^{1/p} (1+\epsilon)^{-1}.
$$
Lemma \ref{estimate}  also implies
$f(m)^{1-\theta} m^{1/\pp} \|\sum_{j=1}^n \alpha_j \psi_j\| \leq
f(M)^{1-\theta}(\sum_{j=1}^n m|\alpha_j|^\pp)^{1/\pp},$
so
$ \|\sum_{j=1}^n \alpha_j \psi_j\| \leq (1+\epsilon)(\sum_{j=1}^n |\alpha_j|^\pp)^{1/\pp}.$
\medskip

%

Clearly $(y_i)_{i=1}^n$ is $(1+\epsilon)$-equivalent to the unit basis of
$\ell_p^n$.
We claim that $Px=\sum_{i=1}^n \psi_i(x)y_i$ defines a projection from $\mathcal F$ onto
$[y_1,\ldots,y_n]$ of norm at most $(1+\epsilon)^{2p}$.
Indeed for $x \in \mathcal F$,
$$\|Px\|^p \leq (1+\epsilon)^p (\sum_{i=1}^n |\psi_i(x)|^p)=
(1+\epsilon)^p (\sum_{i=1}^n \alpha_i |\psi_i(x)|^{p-1}\psi_i(x))$$
for some $\alpha_1,\ldots,\alpha_n$ of modulus 1.
So
$$\|Px\|^p \leq (1+\epsilon)^p \|x\| \|\sum_{i=1}^n \alpha_i |\psi_i(x)|^{p-1}\psi_i\|
\leq  (1+\epsilon)^{p+1} \|x\| (\sum_{i=1}^n |\psi_i(x)^{p-1}|^{\pp})^{1/\pp}.$$
Since
$$\sum_{i=1}^n |\psi_i(x)^{p-1}|^{\pp}=\sum_{i=1}^n |\psi_i(x)|^p \leq (1+\epsilon)^p
\|Px\|^p,$$
we deduce
$\|Px\|^p \leq (1+\epsilon)^{p+1+p/\pp}\|x\| \|Px\|^{p/\pp},$
therefore
$\|Px\| \leq (1+\epsilon)^{2p} \|x\|.$
This concludes the proof of the claim, and up to appropriate choice of $\epsilon$, that
of the proposition.\end{proof}

\section{Iterated twisting of $\mathcal F$}\label{sect:higher-twisting}

The results in this section are the particular cases of \cite[Cor. 2 and Prop.3]{cabecastroch} for the admissible families yielding Ferenczi's space. For the sake of completeness we include a rather complete sketch with somewhat different proofs. To unify the notation, let us set $\mathcal{F}_1 = \mathcal{F}$. As above, $\mathcal{F}_2$
denote the self-extension of $\mathcal{F}_1$ obtained in Section \ref{sect:twisting-Ferenczi}.
As it is showed in Proposition \ref{rep-derived-space},
$$
\mathcal{F}_2 =\{\big(g'(\theta),g(\theta)\big)\,  :\,  g\in\cl{H}(X_{j,t})\},
$$
endowed with the quotient norm of $\cl{H}(X_{j,t})/(\ker \delta_\theta \cap\ker \delta'_\theta)$.
Let us show that the twisting process can be iterated obtaining a sequence $(\mathcal{F}_n)$  of H.I.
spaces such that $\mathcal F_{n+m}$ is a twisted sum of $\mathcal F_n$ and $\mathcal F_m$.
\smallskip


Given a function $g\in\cl{H}(X_{j,t})$ and an integer $k\in\N$, we denote
$\hat g[k]:= g^{(k-1)}(\theta)/(k-1)!,$ the ($k$)-th coefficient of the Taylor series
of $g$ at $\theta$. Following the constructions in \cite{cabecastroch}, we define for $n\geq 3$:
$$
\mathcal{F}_n :=\{\big(\hat g[n],\ldots,\hat g[2],\hat g[1]\big)\,  :\,  g\in\cl{H}(X_{j,t})\}
$$
endowed with the quotient norm of $\cl{H}(X_{j,t})/\bigcap_{k=0}^{n-1}\ker \delta^{(k)}_\theta$.

\begin{prop}
Let $m,n\in\N$ with $m>n$.

\begin{enumerate}
\item  The expression $\pi_{m,n}(x_m,\ldots,x_n,\ldots,x_1) =(x_n,\ldots,x_1)$ defines
a surjective operator $\pi_{m,n} : \mathcal{F}_{m}\to\mathcal{F}_n$.
\smallskip

\item The expression $i_{n,m}(x_n,\ldots,x_1) :=(x_n,\ldots,x_1,0,\ldots,0)$ defines
a isomorphic embedding $i_{n,m} : \mathcal{F}_n\to\mathcal{F}_{m}$ with
${\rm ran}(i_{n,m})=\ker(\pi_{m,m-n})$.
\smallskip

\item The operator $\pi_{m,n}$ is strictly singular.
\end{enumerate}
\end{prop}
\begin{proof}
(1) Since $\dist(g,\bigcap_{k=0}^{n-1}\ker \delta^{(k)}_\theta)\leq
\dist(g,\bigcap_{k=0}^{m-1}\ker \delta^{(k)}_\theta)$, we have $\|\pi_{m,n}\|\leq 1$.
And it is obvious that $\pi_{m,n}$ is surjective.
\smallskip

(2) Let $\phi\in H^\infty(\mathbb{S})$ be a scalar function such that
$\hat \phi[k]= \delta_{k,m-n}$ for $1\leq k \leq m$.
For the existence of $\phi$, we consider a conformal equivalence
$\varphi: \mathbb{S}\to \mathbb{D}$ satisfying $\varphi(\theta)=0$, and the
polynomial $p(z):=(z-\theta)^{m-n}$.
The function $p\circ \varphi^{-1}\in H(\mathbb{D})$ admits a representation
$p\circ \varphi^{-1}(\omega)=\sum_{l=0}^\infty a_l\omega^l$,
and it is not difficult to check that $\phi(z) := \sum_{l=0}^m a_l\varphi(z)^l$
defines a function that satisfies the required conditions.

Given $(x_n,\ldots,x_1)\in \mathcal{F}_n$, we take $g\in\cl{H}(X_{j,t})$ such that
$\hat g[k] = x_k$ for $k=1,\ldots,n$.
Then $f := \phi\cdot g\in \cl{H}(X_{j,t})$ with $\|f\|\leq \|\phi\|_\infty\cdot\|g\|$
and, by the Leibnitz rule,
$$
\hat f[k] = \sum_{l=1}^k \hat\phi[l] \hat g[k-l].
$$
Thus $\hat f[k]=0$ for $1\leq k\leq m-n$ and $\hat f[k]= \hat g[k-m+n]$ for $m-n< k\leq m$;
i.e., $(\hat f[m],\ldots,\hat f[1])=(x_n,\ldots,x_1,0,\ldots,0)$.
Hence $i_{n,m}$ is well-defined and $\|i_{n,m}\|\leq \|\phi\|_\infty$.

Clearly $i_{n,m}$ is injective and ${\rm ran}(i_{n,m})\subset\ker(\pi_{m,m-n})$.
Let $(y_n,\ldots,y_1,0,\ldots,0)$ in $\ker(\pi_{m,m-n})$.
Then there exists $g\in\cl{H}(X_{j,t})$ such that $\hat g[k]=0$ for $1\leq k\leq m-n$ and
$\hat g[k]= y_{k-m+n}$ for $m-n< k\leq m$.
Since $g$ has a zero of order $m-n$ at $\theta$, there exists $f\in\cl{H}(X_{j,t})$ such that
$g(z)= f(z)(z-\theta)^{m-n}$, and it is not difficult to check that
$i_{n,m}(\hat f[n],\ldots,\hat f[1]) =  (y_n,\ldots,y_1,0,\ldots,0)$.
\smallskip

(3) Since $\pi_{m,n}=\pi_{m-1,n}\pi_{m,m-1}$ for $m>n+1$, it is enough to prove that
$\pi_{m,m-1}$ is strictly singular.
We will do it by induction:

We proved in Theorem \ref{ferenczi} that $\pi_{2,1}$ is strictly singular.
Let $m>2$ and assume that $\pi_{m-1,m-2}$ is strictly singular.
Note that $\pi_{m,1}= \pi_{m,2} \pi_{2,1}$; hence $\pi_{m,1}$ is also strictly singular.

We consider the following commuting diagram:
\begin{equation}
\begin{CD}
 0@>>> \mathcal F_{m-1}@>i_{m-1,m}>>  \mathcal F_{m}@>\pi_{m,1}>> \mathcal F_1@>>>0\\
&&@V\pi_{m-1,m-2}VV@VV\pi_{m,m-1}V@|\\
 0@>>>\mathcal F_{m-2}@>>i_{m-2,m-1}>  \mathcal F_{m-1}@>>\pi_{m-1,1}>  \mathcal F_1@>>> 0.
\end{CD}
\end{equation}

By (1) and (2), the two rows are exact.
Suppose that $M$ is an infinite dimensional closed subspace of $ \mathcal F_{m}$
such that $\pi_{m,m-1}|_M$ is an isomorphism.
Since $\pi_{m,m-1} i_{m-1,m}$ is strictly singular and ${\rm ran}(i_{m-1,m})=\ker(\pi_{m,1})$,
$M\cap \ker(\pi_{m,1})$ is finite dimensional and $M+ \ker(\pi_{m,1})$ is closed.
But this is impossible, because $\pi_{m,1}$ is strictly singular.
\end{proof}

Since uniform convexity is a 3-space property \cite{castgonz}, as an immediate consequence we get:

\begin{cor}
Let $m,n\in\N$.
Then the sequence
$$
\begin{CD}
0@>>> \mathcal F_m @>i_{m,m+n}>> \mathcal F_{m+n} @>\pi_{m+n,n}>> \mathcal F_n @>>>0
\end{CD}
$$
is exact and singular. Therefore, all the spaces $\mathcal{F}_n$ are uniformly convex H.I.
\end{cor}

Next we show that there are natural nontrivial twisted sums of spaces $\mathcal{F}_n$
which are not H.I.
Let $l,m,n\in\N$ with $l>n$. We consider the following push-out diagram:

\begin{equation}\label{poz-new}
\begin{CD}
 0@>>> \mathcal F_{l}@>i_{l,l+m}>>  \mathcal F_{l+m}@>\pi_{l+m,m}>> \mathcal F_m@>>>0\\
&&@V\pi_{l,n}VV@VV\pi_{l+m,n+m}V@|\\
 0@>>>\mathcal F_n@>>i_{n,n+m}>  \mathcal F_{n+m}@>>\pi_{n+m,m}>  \mathcal F_m@>>> 0.
\end{CD}
\end{equation}
\vspace{3mm}

\begin{prop}
Let $l,m,n\in\N$ with $l>n$.
Then the diagonal push-out sequence
\begin{equation}\label{diagonal-po}
\begin{CD}
0@>>> \mathcal F_l @>i>> \mathcal F_n\oplus \mathcal F_{l+m} @>\pi>> \mathcal F_{m+n} @>>>0
\end{CD}
\end{equation}
obtained from diagram {\rm (\ref{poz-new})} is a nontrivial exact sequence.
\end{prop}
\begin{proof}
As we saw in Section \ref{sect:twisted-sums}, the maps $i$ and $\pi$ are given by
$$
i(x) =(-\pi_{l,n}\, x,\, i_{l,l+m}\, x)\quad \textrm{ and }\quad \pi(y,z)=
i_{n,n+m}\, y+ \pi_{l+m,n+m}\, z,
$$
and it is easy to check that the sequence (\ref{diagonal-po}) is exact.
Since $l>n$, every operator from $\mathcal F_l$ or $\mathcal F_{m+n}$ into $\mathcal F_n$
is strictly singular.
Thus $\mathcal F_l\oplus \mathcal F_{m+n}$ is not isomorphic to $\mathcal F_n\oplus \mathcal F_{l+m}$,
and the exact sequence (\ref{diagonal-po}) is nontrivial.
%
\end{proof}

\end{document}